\documentclass[10pt]{amsart}
\usepackage{amsmath, amsfonts, amscd, latexsym, amsthm, amssymb}
\usepackage{hyperref}

\newtheorem{Th}{Theorem}[section]
\newtheorem{Lem}[Th]{Lemma}
\newtheorem{Prop}[Th]{Proposition}
\newtheorem{Cor}[Th]{Corollary}

\theoremstyle{Def}
\newtheorem{Def}[Th]{Definition}
\newtheorem{Ex}[Th]{Example}

\theoremstyle{Rem}
\newtheorem{Rem}[Th]{Remark}

\def \c{\mathbb{C}}
\def \z{\mathbb{Z}}
\def \r{\mathbb{R}}
\def \n{\mathbb{N}}
\def \p{\mathbb{P}}

\def \ol{\overline}

\def \L{\mathcal{L}}

\def \K{{\bf K}_{G}(X)}

\def \B{{\bf B}}

\def \s{{\underline{w_0}}}
\def \.{\cdot}

\def \Reg{\textup{Reg}}

\def \ind{\textup{ind}}

\def \Vol{\textup{Vol}}

\def \Lie{\textup{Lie}}

\def \GL{\textup{GL}}

\def \Reg{\textup{Reg}}

\def \conv{\textup{conv}}

\def \ratmap{\dashrightarrow}

\def \GL{\textup{GL}}

\def \SO{\textup{SO}}
\def \SP{\textup{SP}}

\usepackage[margin=1.32in]{geometry}

\begin{document}

\title{Convex bodies associated to actions of reductive groups}

\author{Kiumars Kaveh}
\address{Department of Mathematics, University of Pittsburgh, Pittsburgh, PA 15260}
\email{kaveh@pitt.edu}

\author{Askold G. Khovanskii}
\address{Mathematics, University of Toronto, Toronto, Canada;
Moscow Independent University; Institute for Systems Analysis, Russian Academy of Sciences.}
\email{askold@math.utoronto.ca}


\date{2011 and, in revised form, January 14, 2012.}

\dedicatory{To the memory of  Vladimir Igorevich Arnold}


\begin{abstract}
We associate convex bodies to a wide class of graded $G$-algebras where $G$ is a connected reductive group. These convex bodies give information about the Hilbert function as well as multiplicities of irreducible representations appearing in the graded algebra. We extend the notion of Duistermaat-Heckman measure to graded $G$-algebras and prove a Fujita type approximation theorem and a Brunn-Minkowski inequality for this measure. This in particular applies to arbitrary $G$-line bundles giving an equivariant version of the theory of volumes of line bundles. We generalize the 
Brion-Kazarnowskii formula for the degree of a spherical variety to arbitrary $G$-varieties.
Our approach follows some of the previous works of A. Okounkov. We use the asymptotic theory of semigroups of integral points and Newton-Okounkov bodies developed in \cite{Askold-Kiumars-Newton-Okounkov}.
\end{abstract}

\maketitle




\section*{Introduction} \label{sec-intro}

Discovered in 1975, the famous Bernstein-Kushnirenko theorem gives an elegant formula for the number of solutions of a
generic system of $n$ equations in $(\c^*)^n$ with given Newton polytopes.
The discovery of this theorem was greatly inspired by the rich empirical results which Vladimir Igorevich Arnold had obtained in the course of his research on critical points of functions in several variables.  The Bernstein-Kushnirenko theorem was a starting point for the development of {Newton polytope theory}, which became one of the main subjects studied and developed at the Arnold Seminar. Soon it turned out that this theory is directly related to the theory of toric varieties.

In the early 1980s, the second author suggested the problem of generalizing the Bernstein-Kushnirenko theorem
replacing $(\c^*)^n$ with any reductive algebraic group over $\c$. Such a generalization was found by B. Kazarnovskii (\cite{Kazarnovskii}), { and M. Brion (\cite{Brion1})}.
Unlike the Bernstein-Kushnirenko theorem, Brion-Kazarnovskii's answer is not in terms of {volumes of convex bodies (i.e. a compact convex subset of the Euclidean space).}
In \cite{Okounkov-spherical}, A. Okounkov completes this missing part in the generalization
by constructing (in some important cases) convex bodies
whose volumes give the number of solutions.
Moreover, in \cite{Okounkov-Brunn-Minkowski, Okounkov-log-concave}, Okounkov considers a much more general situation
involving reductive group actions.

Following the work of Okounkov, in \cite{Askold-Kiumars-affine}, \cite{Lazarsfeld-Mustata}
and \linebreak \cite{Askold-Kiumars-Newton-Okounkov},
the authors associate convex bodies to linear systems, and line bundles, on varieties
without requiring presence of a group action.
These convex bodies encode information about the intersection theory of divisors (or linear systems). In particular,
when we have an ample line bundle, the volume of the corresponding body gives the self-intersection
number of the divisor class of the line bundle.



In the present paper we {adopt the general approach of} \cite{Askold-Kiumars-Newton-Okounkov} for
varieties equipped with a reductive group action, extending the original work of Okounkov.
Let $G$ be a connected reductive algebraic group over $\c$. \footnote{Throughout the paper we assume
the ground field to be $\c$, although most of the results hold over an arbitrary algebraically closed field.}
Let $V$ be a finite dimensional $G$-module and $X \subset \p(V)$ a $G$-invariant projective subvariety.
The homogeneous coordinate ring $\c[X] = A = \bigoplus_{k \geq 0} A_{k}$ is a graded $G$-algebra. Fixing {a non-zero element} $\ell \in A_{1}$, we can regard $A$ as the algebra $\bigoplus_{k \geq 0} L^{k}$ where $L$ is
the subspace of rational functions $\{f/\ell \mid f \in A_1\}$.
{(We recall that $L^{k}$ denotes the subspace of $\c(X)$ spanned by all the
products of $k$ elements of $L$.)} Under this identification the
action of $G$ on $A$ corresponds to a {\it twisted action} of $G$ on $\c(X)$ (Section \ref{subsec-inv-subspace}).
In this paper we generalize the above situation. Instead of a projective $G$-variety
we take an arbitrary $G$-variety $X$ of dimension $n$, and instead of the homogeneous coordinate ring, we
consider a graded algebra $A = \bigoplus_{k \geq 0}L_{k}$, with $L_k \subset \c(X)$ for all $k$.
Moreover we assume that $A$ is a $G$-algebra where the action comes from a twisted action of $G$ on $\c(X)$.
We will also assume that the following condition holds:
{\it $A$ is contained in a $G$-algebra $B$ generated
by the constants and finitely many elements in degree $1$.} Note that $A$ itself need not be
a finitely generated algebra. We denote the collection of such algebras by ${\bf A}_{G}(X)$ (Section \ref{subsec-G-alg}).

We would like to mention that, for a normal projective $G$-variety $X$,
the class of algebras in ${\bf A}_G(X)$ already contains all the $G$-linearized graded linear
systems on $X$ (in particular algebras of sections of $G$-linearized line bundles) (see Example
\ref{ex-ring-sections-G-alg}). Thus all the constructions and results
in this paper, in particular, apply to $G$-linearized graded linear systems.
For the purposes of this paper, the authors prefer working with the more general setup of $G$-algebras $A$ in ${\bf A}_G(X)$, instead of graded $G$-invariant linear systems. The reader used to the language of linear systems can think of an algebra $A$ in ${\bf A}_G(X)$ as a $G$-invariant graded linear system on $X$.

To $A$ we associate three convex bodies:
(1) the {\it moment body} $\Delta(A)$,
(2) the {\it multiplicity body} $\widehat{\Delta}(A)$ and,
(3) the {\it string body} $\widetilde{\Delta}(A)$,
together with natural linear projections $\widehat{\pi}: \widehat{\Delta}(A)
\to \Delta(A)$ and $\widetilde{\pi}: \widetilde{\Delta}(A) \to \widehat{\Delta}(A)$.
These bodies encode {information about the asymptotic behavior of} respectively: 
(1) the irreducible representations appearing in the homogeneous components
of $A$, (2) their multiplicities, and (3) the Hilbert function of $A$.

{We should point out that the convex bodies constructed in \cite{Askold-Kiumars-affine}, \cite{Askold-Kiumars-Newton-Okounkov} and \cite{Lazarsfeld-Mustata} do not require a group action and encode information about the Hilbert function. The convex body $\widetilde{\Delta}$ can be considered as a special case of these bodies.}

The constructions of these bodies are based on the notion of convex body associated to
a graded semigroup of integral points (Definition \ref{def-Newton-Okounkov-convex-set}):
Let $S \subset \z_{\geq 0} \times \z^{n}$ be a semigroup. Let $C(S) \subset \r \times \r^{n}$ denote the
cone which is the closure of the convex hull of $S \cup \{0\}$. Let $\pi: \r \times \r^{n} \to \r$ be the {projection onto the first factor}
and suppose $C(S)$ intersects the plane $\pi^{-1}(0)$ only at the origin.
Then the slice $\Delta(S) = C(S) \cap \pi^{-1}(1)$ is a convex body which
we call {\it the Newton-Okounkov body of $S$}. The volume of $\Delta(S)$ is responsible for the asymptotic {behavior} of number of elements
in the level $S_{k} = S \cap \pi^{-1}(k)$, as $k \to \infty$ (Theorem \ref{th-asymp-H_S-vol-Delta}).

Let $A=\bigoplus_{k}L_k \in {\bf A}_{G}(X)$ be a graded $G$-algebra. The {\it moment body} $\Delta(A)$ is the convex body associated
to the semigroup of highest weights: $$S(A) = \{(k, \lambda) \mid V_{\lambda} \textup{ appears in } L_{k}\},$$
where $V_\lambda$ denotes the irreducible representation with highest weight $\lambda$.
When $A$ is finitely generated, $S(A)$ is a finitely generated semigroup and
$\Delta(A)$ is a polytope. This generalizes the notion of moment polytope of
a projective $G$-variety. In Section \ref{sec-moment-body} we prove some basic results about $\Delta(A)$ including a
superadditivity property (Proposition \ref{prop-Delta_G-superadd}). We also {give lower and upper bounds} for
the moment body (Proposition \ref{prop-moment-lower-upper-bound}).

Let $v: \c(X) \setminus \{0\} \to \z^{n}$ be a valuation on the field of rational
functions $\c(X)$, satisfying the conditions in Section \ref{subsec-S-hat}. Following \cite{Okounkov-Brunn-Minkowski}, we define
the {\it multiplicity body}
$\widehat{\Delta}(A)$ to be the convex body associated to the semigroup:
$$\widehat{S}(A) = \bigcup_{k > 0} \{(k, v(f)) \mid f \in L^{U}_{k} \setminus \{0\}\},$$
where $L^{U}_{k}$ is the subspace of unipotent invariants in $L_{k}$.
The multiplicity body $\widehat{\Delta}(A)$ encodes information about the {asymptotic behavior} of multiplicities of irreducible
representations appearing in $A$.
In Section \ref{sec-multi-body}, we prove some basic properties of the multiplicity body including a
superadditivity property (Proposition \ref{prop-hat-Delta_G-superadd}).
We use the construction of $\widehat{\Delta}(A)$
to extend the notion of Duistermaat-Heckman measure of a projective $G$-variety
to graded algebras $A \in {\bf A}_{G}(X)$. Moreover, we prove {a Brunn-Minkowski type inequality} for this measure
(Corollary \ref{cor-Brunn-Mink-Okounkov-algebra}). Section \ref{subsec-Fujita} proves a Fuijta approximation type
theorem for D-H measures, namely the D-H measure of a graded algebra $A$ can be approximated arbitrarily closely
by the usual D-H measures of projective $G$-varieties (Theorem \ref{th-FUjita-D-H-G-algebra}).

In Section \ref{sec-string} we define the {\it string body} $\widetilde{\Delta}(A)$
associated to a graded $G$-algebra $A \in {\bf A}_{G}(X)$. It is the convex body fibered over the
multiplicity body $\widehat{\Delta}(A)$ whose fibers are the {\it string polytopes} of Littelmann and
Berenstein-Zelevinsky associated to irreducible representations of $G$
(see Section \ref{subsec-string-preliminary}). The multiplicity body encodes information
about the {asymptotic behavior} of the Hilbert function of $A$.

The idea of constructing the string body goes back to \cite{Okounkov-spherical} which proposes a similar
construction when $X$ is a spherical variety, using the Gelfand-Cetlin polytopes for $G = \GL(n, \c)$, or other classical groups. (Recall that a $G$-variety is spherical if a Borel subgroup has a dense orbit, see Section \ref{subsec-spherical}.) The idea of using string polytopes, in the case of
spherical varieties of a general reductive group, goes back to \cite{Brion-Alexeev}.

Extending the intersection theory of divisors on complete varieties, in \linebreak \cite{Askold-Kiumars-MMJ} the authors
develop an intersection theory for vector subspaces of the field of rational functions $\c(X)$. Let
$L_1, \ldots, L_n$ be finite dimensional subspaces of $\c(X)$.
The {\it intersection index} $[L_1, \ldots, L_n]$ is defined to be the number of solutions in $X$ of a generic system of
equations $f_1(x) =  \cdots = f_n(x) = 0$ {where $f_i \in L_i$}. When counting the solutions, {we
ignore} the solutions $x$ at which all the functions in some subspace $L_i$
vanish as well as the solutions at which at least one function from some
subspace $L_i$ has a pole. In \cite{Askold-Kiumars-MMJ}, it is shown that this intersection index is well-defined and enjoys properties analogous to
the intersection number of divisors. An important property of the {intersection index is multi-additivity} with respect to a natural product of
subspaces.

In Section \ref{sec-int-index-G} we use the above convex bodies to give formulae for
the intersection indices of invariant subspaces
in terms of volumes or integrals over these bodies
(Theorem \ref{th-self-int-index-G-subspace-vol}, Corollary \ref{cor-self-int-index-G-subspace-integ} and
Corollary \ref{cor-self-int-index-divisor}). These are in fact more concrete and accessible versions of a more general theorem
(\cite[Theorem 4.12]{Askold-Kiumars-Newton-Okounkov}) adapted for the $G$-invariant subspaces.
These in particular give formulae for the intersection numbers of divisors of $G$-linearized line bundles.

It is shown in \cite{Kiumars-string} that
the string convex bodies are special cases of the convex bodies considered in \cite{Askold-Kiumars-affine, Askold-Kiumars-Newton-Okounkov} and \cite{Lazarsfeld-Mustata}, at least when the $G$-variety under
consideration is spherical. In this case, any $G$-algebra $A \in {\bf A}_G(X)$ is multiplicity-free and
the multiplicity body $\widehat{\Delta}(A)$ coincides with
the moment body $\Delta(A)$ (see Section \ref{subsec-spherical}).
Moreover, if the algebra $A$ is finitely generated (e.g. the
homogeneous coordinate ring of a {projective spherical} $G$-variety) we see that the string body of $A$ is a polytope. These
construct a rich class of examples for which it is guaranteed that the convex bodies in \cite{Lazarsfeld-Mustata} and
\cite{Askold-Kiumars-affine} are polytopes.

In Section \ref{subsec-spherical} we address the case of spherical varieties, and recover a
formula of Brion and Kazarnovskii regarding the self-intersection number of divisors on spherical varieties
{(this formula itself is a partial generalization of the Bernstein-Kushnirenko theorem).}\\

\noindent {\bf Acknowledgement:} The authors would like to thank
Valentina Kiritchenko for reading the first draft and giving very useful suggestions, and
Megumi Harada for helpful discussions about the symplectic geometry side of the story.

\section{Preliminaries}
\subsection{Semigroups of integral points and Newton-Okounkov bodies} \label{subsec-semigroup}
In this section we give a brief review of results in \cite{Askold-Kiumars-Newton-Okounkov} on asymptotic
behavior of semigroups of integral points.

Let $S \subset \z^n$ be a semigroup of integral points, that is $S$ is closed under addition. Let:
\begin{itemize}
\item[-] $C(S)$ be the closure of the convex hull of $S \cup \{0\}$, that is, the smallest closed
convex cone (with apex at the origin) containing $S$.
\item[-] $G(S)$, the subgroup of $\z^n$ generated by $S$.
\item[-] $L(S)$, the {vector subspace} of $\r^n$ spanned by $S$.
\end{itemize}
The sets $C(S)$ and $G(S)$ lie in $L(S)$.
To $S$ we associate its {\it regularization} which is the semigroup $\Reg(S) = C(S) \cap G(S)$.
The regularization $\Reg(S)$ is a simpler semigroup with more points and containing the semigroup
$S$. In \cite[Section 1.1]{Askold-Kiumars-Newton-Okounkov} it is proved that {\it $\Reg(S)$
asymptotically approximates $S$.} More precisely:

\begin{Th}[Approximation Theorem] \label{th-approx}
Let $C' \subset C(S)$ be a convex cone which intersects
the boundary (in the topology of the {vector space} $L(S)$) of the cone $C(S)$
only at the origin. Then there exists a constant $N>0$ (depending on $C'$)
such that each point in the group $G(S)$ which lies in
$C'$ and whose distance from the origin is bigger than $N$ belongs to $S$.
\end{Th}

We now consider semigroups in $\z_{\geq 0} \times \z^n$. Let
$\pi: \r \times \r^n \to \r$ denote the projection on the first factor.
Let $S \subset \z_{\geq 0} \times \z^n$ be a semigroup and let
$S_k = S \cap \pi^{-1}(k)$ be the set of points in $S$ at level $k$.
Then $\pi(S)$ consists of $k$ such that $S_k \neq \{0\}$.
It is a subsemigroup in $\z_{\geq 0}$.
Let $m(S)$ be the index of the subgroup generated by $\pi(S)$ in $\z$. For sufficiently
large $k$, we have $k \in \pi(S)$ if and only if $k$ is divisible by $m(S)$.

\begin{Def} \label{def-non-neg-semigroup}
We call a semigroup $S \subset \z_{\geq 0} \times \z^n$ a {\it non-negative semigroup}
if it is not contained in the hyperplane $\pi^{-1}(0)$.
Moreover we assume for simplicity that $m(S) = 1$.
(The assumption $m(S) = 1$ is not crucial and one can slightly modify all the statements
that follow so that they hold without this assumption.)
\end{Def}

As above let $C(S)$ be the
smallest closed convex cone containing $S$, $G(S)$ the subgroup of
$\z^{n+1} = \z \times \z^n$ generated by $S$, and $L(S)$ the rational subspace in
$\r^{n+1}$ spanned by $S$. If in addition the cone $C(S)$ intersects the hyperplane
$\pi^{-1}(0)$ only at the origin,
$S$ is called a {\it strongly non-negative semigroup}.
We denote the group $G(S) \cap \pi^{-1}(0)$ by $\Lambda(S)$ and call it the
{\it lattice associated to the non-negative semigroup $S$}. The index of this sublattice in
$\{0\} \times \z^n$ will be denoted by $\ind(\Lambda(S))$ or simply $\ind(S)$.
Finally, the number of points in $S_k$ is denoted by $H_S(k)$. {The function }$H_S$ is called the
{\it Hilbert function of the semigroup $S$}.


\begin{Def}[Newton-Okounkov convex set] \label{def-Newton-Okounkov-convex-set}
We call the projection of the convex set $C(S) \cap \pi^{-1}(1)$ to $\r^n$ (under the
projection onto the second factor $(1, x) \mapsto x$),
the {\it Newton-Okounkov convex set of the semigroup $S$} and denote it by $\Delta(S)$.
In other words,
$$\Delta(S) = \overline{\conv(\bigcup_{k>0} \{x/k \mid (k, x) \in S_k\})}.$$
If $S$ is strongly non-negative then $\Delta(S)$ is compact and hence a convex body which we call
the {\it Newton-Okounkov body of $S$}.
\end{Def}


Let us define the notion of volume normalized with respect to a lattice.
\begin{Def}[Normalized volume] \label{def-int-volume}
Let $\Lambda \subset \r^n$ be a lattice of full rank $n$.
Let $E \subset \r^n$ be a rational affine subspace of dimension $q$. That is,
$E$ is parallel to a {vector subspace} of dimension $q$ which is rational with respect to $\Lambda$.
The {\it Lebesgue measure normalized with respect to the lattice $\Lambda$} in $E$ is the translation invariant
Lebesgue measure $d\gamma$ in $E$ {normalized so that} the smallest measure of a $q$-dimensional
parallelepiped with vertices in $E \cap \Lambda$ is equal to $1$. The measure of a subset
$A \subset E$ will be called its {\it normalized volume}
and denoted by $\Vol_q(A)$ (whenever the lattice $\Lambda$ is clear from the context).
\end{Def}

Let $\Reg(S)$ be the regularization of $S$ and let $H_{\Reg(S)}$ be its Hilbert function.
It follows from the Approximation Theorem (Theorem \ref{th-approx}) that
$H_S(k)$ and $H_{\Reg(S)}(k)$ have the same {asymptotic behavior} as $k$ goes to infinity.
This implies that the Newton-Okounkov convex set $\Delta(S)$ is responsible for the
{asymptotic behavior} of the Hilbert function of $S$ (see \cite[Theorem 1.13]{Askold-Kiumars-Newton-Okounkov}):

\begin{Th} \label{th-asymp-H_S-vol-Delta}
Let $S$ be a strongly non-negative semigroup.
\begin{enumerate}
\item The function $H_S(k)$ grows like
$a_qk^q$ where $q$ is the dimension of the convex body $\Delta(S)$ and the {\it $q$-th growth coefficient
$a_q = \lim_{k \to \infty} H_S(k)/k^q$}
is equal to $\Vol_q(\Delta(S))$, where $\Vol_q$ is normalized with respect to the lattice $\Lambda(S)$.
\item Let $f: \r^n \to \r$ be a polynomial of degree $d$ and $f^{(d)}$ the
homogeneous component of $f$ of degree $d$. Then
$$\lim_{k \to \infty}\frac {\sum _{(k, x) \in S_{k}}f(x) }{k^{q+d}}
= \int_{\Delta(S)}f^{(d)}(x)d\gamma,$$ where $d\gamma$ is
the Lebesgue measure on $\Delta(S)$ normalized with respect to the lattice $\Lambda(S)$.
\end{enumerate}
\end{Th}

One defines a {\it levelwise addition} operation on the subsets of $\r \times \r^n$:
for each subset $A \subset \r \times \r^n$ and $k \in \r$,
let $A_k$ denote the set of points of $A$ in level $k$, i.e. $A_k = A \cap \pi^{-1}(k)$.
Let $A, B \subset \r \times \r^n$. Define the set $A \oplus_t B \subset \r \times \r^n$ by:
$$A \oplus_t B = \{(x+y, k) \mid k \in \r,~ (x,k) \in A_k,~ (y,k) \in B_k\}.$$

Let $S'$, $S'' \subset \z_{\geq 0} \times \z^n$ be two
non-negative semigroups. One sees that $S=S' \oplus_t S''$ is
a non-negative semigroup, and moreover:
$$\Delta(S) = \Delta(S') + \Delta(S''),$$ where the addition in the right-hand side is the Minkowski
sum of convex sets defined by $X+Y = \{x+y \mid x \in X,~ y \in Y\}$.

\begin{Ex} \label{ex-semigroup-Newton-Okounkov-set}
\noindent 1) Let $S$ be the non-negative semigroup consisting of all the integral points
in $\z_{\geq 0} \times \z$ lying to the right of the broken line $|y| = x$ (where $x$ and $y$ are the first and
second coordinates respectively). The subspace $L$ is the whole $\r^2$.
The cone $C$ is the cone generated by the vectors $(1,1)$ and $(1,-1)$.
The Newton-Okounkov convex body $\Delta(S)$ is the line segment $[-1,1]$.\\

\noindent 2) Let $S$ be the non-negative semigroup consisting of all the integral points
lying to the right of the curve $\sqrt{|y|} = x$. Then the cone $C$ of $S$ is the whole right half-plane
$\{x \geq 0\}$ and thus $S$ is not strongly non-negative. The Newton-Okounkov convex set $\Delta(S)$
is the whole line $\r$ which is unbounded.
\end{Ex}

\subsection{Linear transformations between semigroups} \label{subsec-D-H-semigroup}
Let $T:\r^n \to \r^m$ be a linear map where $n \geq m$. Moreover assume $T(\z^n) \subset \z^m$. 
Let $\tilde{T} = \textup{Id} \oplus T: \r \times \r^n \to \r \times \r^m$,
that is, for $x \in \r^n$, $x_{1} \in \r$ we have $\tilde{T}(x_1, x) = (x_1, T(x))$.

Let $S' \subset \z_{\geq 0} \times \z^n$ be a strongly non-negative
semigroup and $S = \tilde{T}(S')$ its image under $\tilde{T}$. Then $S$ is a {strongly
non-negative semigroup in $\z_{\geq 0} \times \z^m$.}
Let $q = \dim \Delta(S)$ and $q' = \dim \Delta(S')$. We have
$\tilde{T}(L(S')) = L(S)$, $\tilde{T}(C(S')) = C(S)$, $\tilde{T}(G(S')) = G(S)$, and
$T(\Delta(S')) = \Delta(S)$.

For a point $p \in \r^n$, let $\delta_p$ denote the Dirac measure supported at the single point $p$.
{Given $k>0$, define the {\it $k$-th multiplicity measure} $d\mu_{k}$ on $\Delta(S)$ by:}
$$ d\mu_{k} = \sum_{(k, x) \in S_{k}} \#(\tilde{T}^{-1}(k, x) \cap S'_k)\delta_{x/k}.$$
It is a finitely supported measure where a point $z$ has nonzero measure if
$(k, kz) \in S_{k}$, in which case the measure of $z$ is equal to {the number of
points in} $\tilde{T}^{-1}(k, kz) \cap S'_k$. Take a subset
$U_0 \subset \Delta(S)$ and let $U = \{1\} \times U_0$ be its shift to level $1$. We have:
\begin{equation} \label{equ-dmu}
\int_{U_0} d\mu_{k} = \#(\tilde{T}^{-1}(kU) \cap S'_{k}).
\end{equation}
It is clear that the total mass of $d\mu_{k}$ is $\#S'_{k}$.

\begin{Th} \label{th-D-H-semigroup}
The measures $d\mu_{k}/k^{q'}$ weakly converge to a measure $d\mu$ supported on $\Delta(S)$.
Moreover, $d\mu$ is the push-forward of the Lebesgue measure on $\Delta(S')$ to $\Delta(S)$
(normalized with respect to the lattice $\Lambda(S')$).
Thus:
$$\Vol_{q'}(\Delta(S')) = \int_{\Delta(S)} d\mu.$$
\end{Th}
\begin{proof}
We show that the measures $d\mu_{k}/k^{q'}$ converge to the push-forward of the
normalized Lebesgue measure on $\Delta(S')$. To this end, let $U_0 \subset \Delta(S)$ be
a convex open subset which does not intersect the boundary of $\Delta(S)$. Let $U = \{1\} \times U_0$ be the shift
of $U_0$ to the level $1$. Let $U' = \tilde{T}^{-1}(U) \cap C(S')$. It suffices to show that
$$\lim_{k \to \infty}(1/k^{q'})\int_{U_0} d\mu_{k} = \Vol_{q'}(U').$$
One knows that
$$\Vol_{q'}(U') = \lim_{k \to \infty} \frac{\#(kU' \cap G(S'))}{k^{q'}}.$$
{Applying the Approximation Theorem (Theorem \ref{th-approx}) to the semigroup $S'$ we obtain that:}
\begin{eqnarray*}
\Vol_{q'}(U') &=& \lim_{k \to \infty} \frac{\#(kU' \cap S'_k)}{k^{q'}}, \cr
&=& \lim_{k \to \infty} \frac{\int_{U_0} d\mu_{k}}{k^{q'}} \quad \textup{ (from (\ref{equ-dmu})),}
\end{eqnarray*}
which proves the claim.
\end{proof}

Finally, we prove a theorem about the relationship between the measures associated to $(S, S')$
and their finitely generated subsemigroups.
Take an integer $k > 0$ such that $S_k \neq \emptyset$.
Let ${\bf S}_k$ (respectively ${\bf S}'_k$) denote the subsemigroup of
$S$ (respectively $S'$) generated by the level $S_k$ (respectively $S'_k$).
Since $\tilde{T}(S') = S$ we see that $\tilde{T}({\bf S}'_k) = {\bf S}_k$.
Similarly to above, let
$d\rho_k$ denote the measure associated to the pair
$({\bf S}_k, {\bf S}'_k)$, that is, for large $k$:
$$d\rho_k = \lim_{\ell \to \infty} (1/\ell^{q'})
\sum_{(k\ell, x) \in {\bf S}_{k}} \#(\tilde{T}^{-1}(k\ell, x) \cap {\bf S}'_k) \delta_{x/\ell}.$$
(As in Theorem \ref{th-D-H-semigroup} one shows that the above limit of measures exists.)
Note that for large $k$, the subspace $L({\bf S}'_k)$ coincides with $L(S')$ and
the lattice $\Lambda({\bf S}'_k)$ coincides with $\Lambda(S')$.
Let $O_{1/k}: \r^m \to \r^m$ denote multiplication by the scalar $1/k$.
\begin{Th} \label{th-Fujita-D-H-semigroup}
As $k \to \infty$, the measures $O_{1/k}^*(d\rho_k)/k^{q'}$ converge weakly to the measure $d\mu$
associated to the pair $(S', S)$. Here $O_{1/k}^*$ denote the push-forward measure by the map
$O_{1/k}$.
\end{Th}
\begin{proof}
By Theorem \ref{th-D-H-semigroup} we know that the measure $d\mu$ is
the push-forward (under the linear transformation $T$)
of the normalized Lebesgue measure on the body $\Delta(S')$ to $\Delta(S)$.
Let $\Delta'_k$ (respectively $\Delta_k$) denote the convex hull of $\{x' \mid (k, x') \in S'_k\}$
(respectively $\{x \mid (k, x) \in S_k \}$). {As in the proof of Theorem \ref{th-D-H-semigroup}}
one sees that the measure $O_{1/k}^*(d\rho_k)/k^{q'}$ is the push-forward of the normalized Lebesgue measure on
$(1/k)\Delta'_k$ to $(1/k)\Delta_k$.
But as $k \to \infty$, the polytopes $(1/k)\Delta'_k$ (respectively $(1/k)\Delta_k$)
converge to the body $\Delta(S')$ (respectively $\Delta(S)$) with respect to the Hausdorff metric on subsets.
The claim follows easily from this.
\end{proof}

\section{Generalities on reductive group actions} \label{sec-preliminary}
\noindent{\bf Notation:} {Throughout the rest of the paper we will use the following notation: We denote by $G$ a connected reductive algebraic group over $\c$.
A Borel subgroup of $G$ is denoted by $B$ with $T$ and $U$ the maximal torus
and maximal unipotent subgroups contained in $B$ respectively.
The root system of $(G, T)$ is $R$ and $R^+$ denotes the subset of positive
roots for the choice of $B$. We denote by $\alpha_1, \ldots,
\alpha_r$ the corresponding simple roots where $r$ is the
semi-simple rank of $G$. The Weyl group of $(G,T)$ is denoted by $W$ with $w_0 \in W$
its longest element.
The weight lattice of $G$ (that is, the character group of $T$) is denoted by $\Lambda$, and $\Lambda^+$ is the subset
of dominant weights (for the choice of $B$). Put $\Lambda_\r = \Lambda \otimes_{\z} \r$.
Then the convex cone generated by $\Lambda^+$ in $\Lambda_\r$ is the
positive Weyl chamber $\Lambda^+_{\r}$. For a weight $\lambda \in
\Lambda^+$, the irreducible $G$-module corresponding to $\lambda$ will
be denoted by $V_\lambda$ and a highest weight vector in $V_\lambda$ will
be denoted by $v_\lambda$.}

\subsection{$G$-varieties and invariant subspaces of rational functions} \label{subsec-inv-subspace}
Let $X$ be an irreducible $G$-variety, that is $X$ is equipped with an algebraic action of $G$. We will denote the
action of $g \in G$ on $x \in X$ by $g \cdot x$. The group $G$ acts on $\c(X)$, the field of rational functions on $X$,
by $(g \cdot f)(x) = f(g^{-1} \cdot x)$, for $g \in G$ and $f \in \c(X)$.
We refer to this as the {\it natural action of $G$ on $\c(X)$}.
The above action of $G$ on $\c(X)$ restricts to an
action of $G$ on the ring of regular functions $\mathcal{O}(X)$.
With this action, the ring $\mathcal{O}(X)$ is a {\it rational $G$-module},
that is every $f \in \mathcal{O}(X)$ lies in a finite dimensional $G$-submodule.
In general the field $\c(X)$ is not a rational $G$-module.

Let $X$ be any irreducible $G$-variety (not necessarily projective) with the
field of rational functions $\c(X)$. Let $\varphi: G \to \c(X)^* := \c(X) \setminus \{0\}$
be a map satisfying the following condition: for all $g_1, g_2 \in G$ we have
$$\varphi_{g_1g_2} = (g_1 \cdot \varphi_{g_2})\varphi_{g_1}.$$ Moreover, assume that
for every $x \in X$, $g \mapsto \varphi_g(x)$ gives a rational function on $G$.
Such a function is called a {\it group cocycle}. The cocycles form a group under multiplication.
Given a group cocycle as above one defines an {\it action of
$G$ on $\c(X)$ twisted by the cocycle $\varphi$}: Let $f \in \c(X)$ and $g \in G$, define the
$\varphi$-twisted action $g *_\varphi f$ by:
$$g *_\varphi f = (g \cdot f) \varphi_g,$$ where in the right-hand side we have
the multiplication of rational functions $g \cdot f$ and $\varphi_g$.

\begin{Ex} \label{ex-G-line-bundle-cocycle}
Let $\L$ be a $G$-linearized line bundle on a normal $G$-variety $X$.
Fix a nonzero section $\tau \in H^0(X, \L)$ with divisor $D$ {(throughout the paper we work only with
line bundles which have nonzero global sections, although this assumption is not crucial and the statements can be slightly modified to hold in a more general setting).} One verifies that $\varphi: G \to \c(X)^*$
defined by $\varphi_g = (g \cdot \tau) / \tau$ is a cocycle. Let $L(D) = \{ f \in \c(X) \mid (f) + D > 0 \}$ be the subspace associated to the divisor $D$. Under the identification $L(D) \cong H^0(X, L)$ given by $f \mapsto f \tau$,
the $\varphi$-twisted action of $G$ on $L(D) \subset \c(X)$ corresponds to the action of $G$ on $H^0(X, \L)$.
\end{Ex}

Let $Z(G, X)$ denote the group of all cocycles $\varphi: G \to \c(X)^*$ and let
$B(G, X)$ denote the subgroup of {\it coboundaries}, i.e.
cocycles of the form $\varphi_g = g \cdot h/ h$ for some function $h \in \c(X)^*$.
The quotient $Z(X, G)/B(X, G)$ is usually called the {\it first group cohomology of
$G$ with coefficients in $\c(X)^*$} and denoted by $H^1(G, \c(X)^*)$.

\begin{Def}
Let $L \subset \c(X)$ be a finite dimensional subspace of rational functions.
If $L$ is stable under the action of $G$ twisted by a cocycle $\varphi$, we call
$L$ a {\it $\varphi$-invariant subspace}. We denote the collection of
all the pairs $(L, \varphi)$, where $L$ is a finite dimensional
$\varphi$-invariant subspace, by $\K$. By abuse of terminology, we call a pair
$(L, \varphi) \in \K$ an {\it invariant subspace}.
\end{Def}


If $L$ is a $\varphi$-invariant subspace and $h \in \c(X)^*$
then {the subspace $h^{-1}L$ is $\varphi'$-invariant} where $\varphi' = (g\cdot h/ h) \varphi$.

For two subspaces $L_1$, $L_2$ of rational functions, let $L_1L_2$ denote the subspace
spanned by all the products $f_1f_2$ where $f_1 \in L_1$ and $f_2 \in L_2$.
The following is straightforward:
\begin{Prop}
Let $(L_1, \varphi_1), (L_2, \varphi_2) \in \K$, then $(L_1L_2, \varphi_1 \varphi_2) \in \K$.
That is, $\K$ is a semigroup with respect to this multiplication of pairs.
\end{Prop}

To a finite dimensional subspace $L$ of rational functions one can assign its Kodaira map
$\Phi_L$ which is a rational map from $X$ to $\p(L^*)$, the projectivization of the dual space $L^*$:
Let $x \in X$ be such that $f(x)$ is defined for all $f \in L$. Then $\Phi(x)$ is represented by the linear functional
in $L^*$ which sends $f$ to $f(x)$. 

Now let $(L, \varphi) \in \K$ be an invariant subspace and $\Phi_L: X \ratmap \p(L^*)$ its Kodaira map.
The $\varphi$-twisted action of $G$ on $L$ induces an action of $G$ on $L^*$ and hence
on $\p(L^*)$. We have the following:
\begin{Prop} \label{prop-Kodaira-map}
The Kodaira map $\Phi_L$ is $G$-equivariant.
\end{Prop}

Let $L_1, \ldots, L_n \subset \c(X)$ be nonzero finite dimensional subspaces of rational functions.
The {\it intersection index} $[L_1,\dots, L_n]$ is the number of solutions in $
X$ of a generic system of equations $f_1=\dots=f_n=0$, where
$f_1\in L_1,\dots, f_n\in L_n$. In counting the solutions, we
neglect the solutions $x$ at which all the functions in some space $L_i$
vanish as well as the solutions at which at least one function from some
space $L_i$ has a pole.
In \cite{Askold-Kiumars-MMJ} it is proven that the intersection index of subspaces is well-defined and is
{multi-additive} with respect to the product of subspaces. We view the intersection index of subspaces of rational
functions as an extension of the intersection theory of (Cartier) divisors (more generally linear systems).

Let $Y_L$ denote the closure of the image of the Kodaira map $\Phi_L$. From definition one sees that
if $\dim(Y_L) < n$, then the self-intersection index $[L, \ldots, L]$ is equal to zero. On the other hand, if
$\dim(Y_L) = n$ then the self-intersection index $[L, \ldots, L]$ is equal to the mapping degree of $\Phi_{L}$ times the degree of
$Y_L$ as a subvariety of the projective space $\p(L^*)$.

\begin{Def} \label{def-A_L}
To an invariant subspace $(L, \varphi)$ we associate the graded algebra
$$A_L = \bigoplus_{k \geq 0} L^k.$$ The group $G$ acts on $A_L$ by acting on each $L^k$ via $\varphi^k$.
One verifies that $A_L$ is isomorphic, as a graded $G$-algebra, to
the homogeneous coordinate ring of the projective variety $Y_L$.
\end{Def}


\subsection{Graded $G$-algebras} \label{subsec-G-alg}
An algebra $A$ is called a {\it $G$-algebra} if $G$ acts on $A$ respecting the algebra operations.
If $A$ is graded we require that {the action of $G$ respects} the grading.
Let $X$ be a $G$-variety with the field of rational functions $F = \c(X)$.
We will deal with graded algebras $A = \bigoplus_{k \geq 0} L_k$, where for each $k$, the $k$-th
homogeneous component $L_k$ is a subspace of $F$ {(one can regard
such $A$ as a graded subalgebra of the ring of polynomials $F[t]$).}

\begin{Def}
Let $A = \bigoplus_{k \geq 0} L_k$, $L_k \subset F$ be a graded algebra. 
We say $A$ is a {\it graded $G$-algebra with $G$-action twisted by $\varphi$},
if for each $k > 0$, the subspace $L_k$ is $\varphi^k$-invariant.
Moreover, we let ${\bf A}_G(X)$ denote the collection of all the pairs
$(A, \varphi)$ satisfying the following conditions:
\begin{itemize}
\item[(1)] $A$ is a graded $G$-algebra with the action twisted by $\varphi$.
\item[(2)] There is an invariant subspace $(L, \varphi) \in {\bf K}_G(X)$ such that
$A \subset A_L$. In particular, for every $k\geq 0$ we have $\dim L_k < \infty$.
\item[(3)] For large $k$, $L_k \neq \{0\}$.
\end{itemize}
By abuse of terminology we call a pair $(A, \varphi) \in {\bf A}_G(X)$ a {\it graded algebra of almost
$G$-integral type}. {Assumption (3) is not crucial} and is made to make the statements in
the rest of paper simpler. One can slightly modify the statements so that they hold without this assumption.
\end{Def}


\begin{Ex} \label{ex-ring-sections-G-alg}
1)($G$-algebra associated to an invariant subspace)
Let $(L, \varphi) \in {\bf K}_G(X)$ be an invariant subspace of rational functions on
$X$. Then $(A_L, \varphi) \in {\bf A}_G(X)$.\\

\noindent 2)(Integral closure) Let $(A, \varphi) \in {\bf A}_G(X)$ be a graded
algebra of almost $G$-integral type. Let $\overline{A}$ denote the integral closure of
$A$ in $F(t)$, where we regard $A$ as a graded subalgebra of the polynomial ring $F[t]$.
One verifies that $(\overline{A}, \varphi) \in {\bf A}_G(X)$.
For a subspace $(L, \varphi) \in \K$, in Section \ref{sec-int-index-G}, the integral closure $\ol{A_L}$ is used
in computing the self-intersection index of $L$. The algebra $\ol{A_L}$ and its
connection with the so-called completion $\ol{L}$ of the subspace $L$ is discussed in detail in \cite[Section 6]{Askold-Kiumars-MMJ}. {This is related to the notion of a complete linear system.}\\


\noindent 3)(Algebra of sections of a $G$-line bundle) Let $X$ be a normal
projective irreducible $G$-variety and let $\L$ be a
$G$-linearized line bundle on $X$. To $\L$ one associates the {\it algebra of sections}:
$$R(\L) = \bigoplus_{k \geq 0} H^0(X, \L^{\otimes k}).$$
Fix a nonzero section $\tau \in H^0(X, \L)$
with divisor $D$. As in Example \ref{ex-G-line-bundle-cocycle}, one identifies the algebra $R(\L)$
with the algebra $$R(D) = \bigoplus_{k \geq 0} L(kD),$$ where
$L(kD) = \{ f \in \c(X) \mid (f) + kD > 0 \}.$ With abuse of terminology we also call $R(D)$ the algebra of
sections of $D$.
Let the cocycle $\varphi: G \to \c(X)^*$
be defined by $\varphi_g = (g \cdot \tau) / \tau$.
Then the action of $G$ on $R(\L)$ corresponds to the $\varphi$-twisted action of $G$ on $R(D)$.
One also shows that $(R(D), \varphi)$ is a graded algebra of almost $G$-integral type, i.e.
$(R(D), \varphi) \in {\bf A}_G(X)$.
Suppose $\L$ is very ample and let $L = L(D)$. Then one can identify
the algebra of sections $R(D)$ and the integral closure $\ol{A_L}$ of the homogeneous coordinate ring $A_L$
(see \cite[Chap. 2, Proof of Theorem 5.19]{Hartshorne}).
\\

\noindent 4)(Componentwise product of algebras)
Let $A' = \bigoplus_{k} L'_k$, $A'' = \bigoplus_k L''_k$ be two graded algebars with $L'_k, L''_k \subset F$.
One defines the {\it componentwise product} $A'A''$ to be the graded algebra $A'A'' = \bigoplus_{k} L'_kL''_k.$
Now let $(A', \varphi')$, $(A'', \varphi'') \in {\bf A}_G(X)$.
It is easy to see that  $(A'A'', \varphi'\varphi'') \in {\bf A}_G(X)$, i.e. the componentwise product $A'A''$ is
a $G$-algebra with the action twisted by $\varphi'\varphi''$ and is of almost $G$-integral type.
\end{Ex}

\section{Moment convex body of a $G$-algebra} \label{sec-moment-body}
\subsection{Semigroup of highest weights and moment convex body} \label{subsec-semigp-S}
In this section we discuss the semigroup of highest weights and moment convex body for a graded $G$-algebra $(A, \varphi) \in {\bf A}_G(X)$, where as before $X$ is an irreducible $G$-variety.

One associates a lattice and a semigroup of weights to $X$ which measure
the highest weights that can appear in the $G$-modules consisting of
functions on $X$.

\begin{Def}
For any variety $X$ let $\Lambda(X) \subset \Lambda$ denote the lattice of weights of $B$-eigenfunctions
of $\c(X)$ for the natural action of $G$. It is called the {\it weight lattice of $X$}.
Also if $X$ is quasi-affine let $\Lambda^+(X)$ denote the semigroup of
weights of $B$-eigenfunctions in the algebra of regular functions $\mathcal{O}(X)$.
It is called the {\it weight semigroup of $X$}.
The rank of the lattice $\Lambda(X)$ is denoted by $r(X)$ and is called the {\it rank of the $G$-variety $X$}.
The linear span of $\Lambda(X)$ will be denoted by $\Lambda_\r(X)$.
\end{Def}


Now let $A = \bigoplus_{k \geq 0} L_k$ be a graded $G$-algebra. Let us write $L_k$ as the sum of its isotypic components $$L_k = \bigoplus_{\lambda \in \Lambda^+} L_{k,\lambda},$$ where
$L_{k, \lambda}$ is the sum of all the copies of the irreducible representation $V_\lambda$ in $L_k$.
The following statement about the product of isotypic components is well-known:

\begin{Th}[Multiplication of isotypic components] \label{th-product-isotypic}
Let $\lambda$, $\mu$ be dominant weights and $k, \ell >0$.
If $L_{k, \lambda}$ and $L_{\ell, \mu}$ are nonzero then $V_{\lambda+\mu}$ appears in
$L_{k,\lambda}L_{\ell,\mu}$ with nonzero multiplicity. Moreover,
if $V_{\nu}$ appears in $L_{k, \lambda} L_{\ell, \mu}$ with nonzero multiplicity
then $\nu$ lies in the convex hull of the $W$-orbit of $\lambda+\mu$
intersected with the positive Weyl chamber $\Lambda^+_\r$.
\end{Th}

\begin{Def}[Semigroup of highest weights and moment convex set]
\begin{enumerate}
\item Define the set \linebreak $S_G(A) \subset \z_{\geq 0} \times \Lambda^+$ by:
$$S_G(A) = \{(k, \lambda) \mid L_{k, \lambda} \neq \{0\}\}.$$
From Theorem \ref{th-product-isotypic}, $S_G(A)$ is
a semigroup (under addition). We call $S_G(A)$ the {\it semigroup of highest weights of $A$} or simply
{\it weight semigroup of $A$}. When there is no chance of confusion we will
drop the subscript $G$ and write $S(A)$ instead of $S_G(A)$.
\item We call the Newton-Okounkov convex set
of the semigroup $S(A)$
the {\it moment convex set} or simply the {\it moment set} of $A$
and denote it by $\Delta_G(A)$ or simply $\Delta(A)$.
\item We call the lattice $\Lambda(S(A))$ associated to the semigroup $S(A)$, the {\it
weight lattice of $A$} and denote it by $\Lambda(A)$. Also {the vector subspace} spanned by
$\Lambda(A)$ will be denoted by $\Lambda_\r(A)$.
\end{enumerate}
\end{Def}

\begin{Rem} \label{rem-S_T(A^U)}
Consider the subalgebra $A^U$ of $U$-invariants in $A$.
Since the action of $U$ respects the grading we have $A^U = \bigoplus_{k \geq 0}L_k^U$.
Because $T$ normalizes $U$, the algebra $A^U$ is stable under the action of $T$, i.e. is a $T$-algebra.
One can alternatively define $S(A)$ to be the semigroup of $T$-weights of the subalgebra
$A^U$, i.e. $S_T(A^U) = S_G(A)$.
\end{Rem}

\begin{Rem}[Connection with moment polytope in symplectic geometry] \label{rem-moment-polytope-symplectic}
{Let $K$ be a compact Lie group and let $X$ be a compact Hamiltonian $K$-manifold} with the moment map
$\mu: X \to \Lie(K)^*$. It is a well-known result due to F. Kirwan
that the intersection of the image of the moment map with the positive Weyl chamber is a
convex polytope usually called the {\it moment polytope} or {\it Kirwan polytope} of the
Hamiltonian $K$-space $X$.

Let $(L, \varphi) \in {\bf K}_G(X)$ be an invariant subspace of rational functions with the
Kodaira map $\Phi_L: X \to \p(L^*)$. Let $Y_L \subset \p(L^*)$ denote the
closure of the image of the Kodaira map $\Phi_L$. It is a closed irreducible $G$-invariant subvariety of
the projective space $\p(L^*)$ with homogeneous coordinate ring $A_L$. Let us assume that $Y_L$ is smooth.
Fix a $K$-invariant inner product on $L^*$ {where $K$ is a
maximal compact subgroup of $G$}. This induces a $K$-invariant symplectic structure on $\p(L^*)$ and
hence on $Y_L$.
\begin{enumerate}
\item With this symplectic structure, $Y_L$ is a Hamiltonian $K$-manifold.
\item {From the principle} {\it quantization commutes with reduction}
it follows that this convex polytope coincides with $\Delta(A_L)$ (see
\cite{G-S} and \cite{Brion2}). {More precisely, the Kirwan polytope identifies with $\Delta(A_L)$ after taking the involution $\lambda \mapsto \lambda^*$, where $\lambda^* = -w_0 \lambda$.}
\end{enumerate}
\end{Rem}

The following proposition shows that the weight lattices of all the graded $G$-algebras are in fact
contained in the weight lattice of $X$. We skip the proof.
\begin{Prop} \label{prop-Lambda(A)}
1) Let $(A, \varphi)$ be a graded $G$-algebra. Then the lattice of weights
$\Lambda(A)$ associated to $A$ is contained in the lattice of weights $\Lambda(X)$.
It follows that the moment body $\Delta(A)$ is {parallel to the subspace} $\Lambda_\r(X)$.
2) Suppose $L \in {\bf K}_G(X)$ is such that the Kodaira map $\Phi_L$ gives a birational
isomorphism between $X$ and its image. Then the weight lattice $\Lambda(A_L)$ coincides
with $\Lambda(X)$.
\end{Prop}

The following is well-known (see \cite{Popov}):
\begin{Th} \label{th-S-finitely-gen}
1) If $A$ is a finitely generated $G$-algebra then $A^U$ is also finitely generated.
2) If $A$ is a finitely generated graded $G$-algebra then the semigroup $S(A)$ is a finitely
generated semigroup.
\end{Th}

\begin{Cor} \label{cor-S(A)-strongly-non-negative}
1) If $A$ is a finitely generated $G$-algebra then the moment convex body $\Delta(A)$ is
in fact a polytope which we call the {\it moment polytope of $A$}. In particular
for any invariant subspace $(L, \varphi) \in {\bf K}_G(X)$, the moment convex set $\Delta(A_L)$ is a
polytope.
2) Let $(A, \varphi) \in {\bf A}_G(X)$, i.e. $A$ is an algebra of almost $G$-integral type. Then
$S(A)$ is a strongly non-negative semigroup. It follows that the moment convex set $\Delta(A)$ is a
convex body which we call the {\it moment convex body} or simply the {\it moment body of $A$}.
\end{Cor}
\begin{proof}
1) Follows from Theorem \ref{th-S-finitely-gen}(2). 2) Let $(A, \varphi) \in {\bf A}_G(X)$ be an algebra
of almost $G$-integral type.
From definition there is an invariant subspace $(L, \varphi) \in {\bf K}_G(X)$ such that
$A \subset A_L$. By 1) we know that
$S(A_L)$ is finitely generated. As none of the generators lie in $\Lambda_\r \times \{0\}$, the semigroup
$S(A_L)$ is a strongly non-negative semigroup. Now since $S(A) \subset S(A_L)$ we conclude that $S(A)$ is
also a strongly non-negative semigroup.
\end{proof}

The following example concerns the situation in the Bernstein-Kushnirenko theorem on the number of
solutions of a system of Laurent polynomials in $(\c^*)^n$.
\begin{Ex}[$T$-invariant subspaces of Laurent polynomials] \label{ex-moment-body-toric}
Let $T = X = (\c^*)^n$ be the algebraic torus acting on itself by multiplication.
In $T$, the weight lattice $\Lambda$ and the semigroup of dominant weights $\Lambda^+$ both coincide with the lattice $\chi(T)$ of characters of $T$.
Fixing a set of coordinates $(x_1, \ldots, x_n)$ on $X$ we identify $\chi(T)$ with $\z^n$.
Let $I \subset \z^n$ be a finite set of characters. To $I$ we can associate a subspace $L(I)$ consisting of
all the Laurent polynomials with exponents in $I$. The subspace $L(I)$ is a $T$-invariant finite dimensional subspace of rational functions on $X$ for the natural action of $T$. The dimension of $L(I)$ is equal to $\#I$. Let $A = A_{L(I)}$ be the corresponding graded algebra. One sees that the semigroup of weights $S(A)$ is the semigroup in $\z_{\geq 0} \times \z^n$ generated by $\{1\} \times I$ and the moment convex body $\Delta(A)$
is the convex hull of the points in $I$.
\end{Ex}

From another point of view, the above example describes the moment polytope of an equivariant
very ample line bundle on a toric variety.

The next example concerns the generalized Pl\"{u}cker embedding of the flag variety into a projective space.
\begin{Ex}[Line bundles on flag variety] \label{ex-moment-body-G/B}
1) Let $X = G/B$ be the complete flag variety of $G$.
Let $\lambda \in \Lambda^+$ be a dominant weight and let $\L_\lambda$ denote the
$G$-linearized line bundle on $X$ associated to $\lambda$. {By the Borel-Weil-Bott theorem, for any integer $k>0$,
the $G$-module $H^0(X, \L_\lambda^{\otimes k})$ is isomorphic to $V_{k\lambda}^*$.}
Let $\tau_\lambda$ be a highest weight section in $H^0(X, \L_\lambda)$
with divisor $D_\lambda$. As in Example \ref{ex-ring-sections-G-alg}(3) consider the algebra:
$$A = R(D_\lambda) = \bigoplus_{k \geq 0} L(kD_\lambda) \cong \bigoplus_{\geq 0} V_{k\lambda}^*.$$
Let us regard $A$ as a graded $G$-algebra. For a dominant weight $\lambda$ let $\lambda^*$ be the dominant weight $-w_0\lambda$. Then one knows that $V_{k\lambda}^* \cong V_{k\lambda^*}$ (as $G$-modules) and hence the semigroup $S_G(A)$ is given by $$S_G(A) = \{ (k\lambda^*, k) \mid k \in \n \}.$$ Thus the moment convex body $\Delta_G(A)$
coincides with the single point $\{ \lambda^* \}$.\\

\noindent 2) On the other hand let us regard the algebra $A$ above as a $T$-algebra and
describe the moment body $\Delta_T(A)$. It is well-known that the convex hull of $T$-weights in
$V_\lambda$ coincides with the convex hull of the Weyl group orbit $\{w\lambda \mid w \in W\}$. It thus
follows that the moment convex body $\Delta_T(A)$ is the convex hull of the Weyl group orbit
$\{w\lambda^* \mid w \in W \}$.
\end{Ex}

The following is related to the situation addressed in \cite{Kazarnovskii}
generalizing the Bernstein-Kushnirenko theorem to representations of reductive groups.
(\cite{Kazarnovskii} uses a {symplectic geometry} approach.)
\begin{Ex}[Moment polytope of a representation]
{Let $\pi: H \to \GL(n, \c) \subset \textup{Mat}(n, \c)$ be a
finite dimensional representation of a connected reductive group $H$.
Let $\pi_{ij}: H \to \c$, $i,j=1, \ldots, n$ be the matrix elements, i.e. the
entries of $\pi$. Let $L$ be the subspace of regular functions on $H$ spanned by the
$\pi_{ij}$. Consider the action of $G = H \times H$ on $H$ by the multiplication
from left and right. The subspace $L$ is a $G$-invariant subspace.
Let $\Lambda_\r^+$ (respectively $W$) denote the positive Weyl chamber (respectively the Weyl group) of $H$.
As in \cite{Kazarnovskii}, one can show that the moment polytope of the $G$-algebra $A_L$
(which lives in $\Lambda_\r^+ \times \Lambda_\r^+$)
can be identified with the convex hull of the $W$-orbit of the highest weights of the representation $\pi$
intersected with the positive Weyl chamber $\Lambda_\r^+$.}
\end{Ex}

We now give lower and upper estimates for the moment polytope of an algebra $A_L$
associated to an invariant subspace $(L, \varphi) \in {\bf K}_G(X)$. Note that
$A_L$ is finitely generated and hence its moment body is a polytope
(Corollary \ref{cor-S(A)-strongly-non-negative}).
We define two polytopes which are lower and upper
estimates for the moment polytope $\Delta(A_L)$.
Let $I \subset \Lambda^+$ be the finite set of all dominant weights $\lambda$
where $V_\lambda$ appears in the $G$-module $L$ with nonzero multiplicity.
Let $P(I)$ be the convex hull of $I$. It is a convex polytope in $\Lambda^+_\r(X)$.
Also let $P_W(I)$ be the intersection of $\Lambda^+_\r(X)$ with the convex polytope
obtained by taking the convex hull of the union of $W$-orbits of all the $\lambda \in I$.

\begin{Prop} \label{prop-moment-lower-upper-bound}
With notation as above, we have
$$P(I) \subset \Delta(A_{L}) \subset P_W(I).$$
\end{Prop}
\begin{proof}
Follows immediately from Theorem \ref{th-product-isotypic}.
\end{proof}

There are two well-known classes of $G$-varieties for which the
two extremes of $P(I)$ and $P_W(I)$ are obtained.
For the homogeneous coordinate ring of a so-called {\it horospherical $G$-variety} the moment polytope
is equal to $P(I)$. On the other hand, the upper estimate $P_W(I)$ is attained for {\it symmetric varieties}.
It turns out that in both cases the moment polytope is additive, in the sense which is explained in Section \ref{subsec-superadd-moment} (see \cite{Kiumars-note-spherical} for the case of symmetric varieties).

\subsection{Superadditivity of the moment convex body} \label{subsec-superadd-moment}
Let $X$ be an irreducible $G$-variety.
In this section we address the additivity of the mapping $A \mapsto \Delta(A)$ with respect to the
componentwise product of $G$-algebras in ${\bf A}_G(X)$.
We will show that in general this map is superadditive but in the case of a torus action $G=T$
is additive.

Let us start with the case of a torus action. Let $T$ be a torus and $X$ an irreducible $T$-variety.
\begin{Prop} \label{prop-T-moment-polytope-alg-additive}
1) Let $(A', \varphi'), (A'', \varphi'') \in {\bf A}_T(X)$ be two graded $T$-algebras. Then
$\Delta_T(A) = \Delta_T(A') + \Delta_T(A'')$, where
$(A, \varphi) = (A'A'', \varphi'\varphi'')$ is the
componentwise product of $A'$, $A''$, and the addition  in the right-hand side is the
Minkowski sum of convex bodies.
2) Let $(L, \varphi)$ be a $T$-invariant subspace of rational functions and let $I$ be the finite set
of $T$-weights of $L$. Then the moment convex body of the algebra $A_L$ coincides with the convex hull of
$I$. 3) The map $L \mapsto \Delta_T(A_L)$ is additive with
respect to the multiplication of subspaces. That is, if $L_1, L_2$
are two $T$-invariant subspaces then $$\Delta_T(A_{L_1L_2}) = \Delta_T(A_{L_1}) + \Delta_T(A_{L_2}).$$
\end{Prop}
\begin{proof}
1) If $f$ and $g$ are $T$-eigenfunctions of weights $\lambda$ and $\gamma$ respectively then $fg$ is a
$T$-eigenfunction of weight $\lambda+\gamma$. It implies that
the semigroup $S_T(A)$ is the levelwise addition of semigroups $S_T(A')$ and $S_T(A'')$. It then
follows that $\Delta_T(A) = \Delta_T(A')+\Delta_T(A'')$ (see Section \ref{subsec-semigroup}).
2) The semigroup $S(A_L)$ is generated by the finite set
$\{1\} \times I \subset \z_{\geq 0} \times \chi(T)$ and hence its moment convex body coincides with the
convex hull of $I$.
3) Follows immediately from 1).
\end{proof}

Now we address the case of a reductive group action.
Let $(A', \varphi), (A'', \varphi'') \in {\bf A}_G(X)$ be two graded
$G$-algebras. We have the following inclusion of the algebras of $U$-invariants:
$$(A'^U)(A''^U) \subset (A'A'')^U.$$
In general $(A'^U)(A''^U)$ might be strictly
smaller than $(A'A'')^U$, and hence the map $A \mapsto \Delta(A)$ is
only superadditive with respect to the componentwise product of algebras (see Remark \ref{rem-S_T(A^U)}):
\begin{Prop}[Superadditivity] \label{prop-Delta_G-superadd}
With notation as above, we have
$$\Delta(A') + \Delta(A'') \subset \Delta(A).$$
In particular, if $(L', \varphi'), (L'', \varphi'') \in \K$ are invariant subspaces then we have
$$\Delta(A_{L'}) + \Delta(A_{L''}) \subset \Delta(A_{L'L''}).$$
\end{Prop}

\section{Multiplicity convex body and Duistermaat-Heckman measure} \label{sec-multi-body}
\subsection{Multiplicity convex body of a $G$-algebra} \label{subsec-S-hat}
As usual let $X$ be an irreducible $G$-variety with $\dim X = n$.
Following Okounkov \cite{Okounkov-Brunn-Minkowski}, in this section
we define a larger semigroup $\widehat{S}(A)$ lying over the weight semigroup
$S(A)$ such that it encodes information about the multiplicities of irreducible representations appearing in $A$.

In \cite{Okounkov-Brunn-Minkowski} the author deals with
the homogeneous coordinate rings of projective $G$-varieties.
Here we deal with the larger class of graded algebras of almost $G$-integral type.


Equip $\z^n$ with a total order which respects addition. Fix a valuation $v: \c(X)^* \to \z^n$
with the following properties: 1) The image of $v$ is the whole $\z^n$, 2) $v$ has {\it one-dimensional leaves}, i.e.
if for $f, g \in \c(X)^*$, $v(f) = v(g)$ then there is $c \in \c$ with $v(f - cg) < v(f)$ or $f-cg = 0$, 3) $v$ is invariant under the action of the Borel subgroup $B$. 

{ When $X$ is a projective $G$-variety, from the Lie-Kolchin theorem it follows that a valuation with the above properties exists (see \cite{Okounkov-Brunn-Minkowski}). On the other hand, if $X$ is a normal $G$-variety, by a theorem of Sumihiro, $X$ is $G$-equivariantly birational to a projective $G$-variety.}

\begin{Def}
Let $\widehat{\Lambda}(X) = v(\c(X)^U \setminus \{0\})$ be the image of the subfield of
rational $U$-invariants under the valuation $v$. It is a lattice in $\z^n$. The rank
of $\widehat{\Lambda}(X)$ is equal to the transcendence degree of the field $\c(X)^U$.
We will denote this rank by $\widehat{r} = \widehat{r}(X)$. We also denote the linear span of $\widehat{\Lambda}(X)$ by
$\widehat{\Lambda}_\r(X)$.
\end{Def}

Let $(A, \varphi)$ be a graded $G$-algebra and write $A = \bigoplus_{k \geq 0} L_k$.
\begin{Def}
1) Define the set $\widehat{S}_G(A) \subset \z_{\geq 0} \times \z^n$ by
$$\widehat{S}_G(A) = \bigcup_{k \geq 0} \{ (k, v(f)) \mid ~f \in L_{k}^U \setminus \{0\}\}.$$
We call the semigroup $\widehat{S}_G(A)$, the {\it multiplicity semigroup of the algebra $A$}.
When there is no ambiguity we will simply write $\widehat{S}(A)$ instead of $\widehat{S}_G(A)$.
2) Denote the Newton-Okounkov convex set of the semigroup $\widehat{S}_G(A)$ by
$\widehat{\Delta}_G(A)$ or simply $\widehat{\Delta}(A)$ and call it the {\it multiplicity convex set of $A$}.
3) Denote the lattice $\Lambda(\widehat{S}(A))$ associated to the semigroup $\widehat{S}(A)$ by
$\widehat{\Lambda}(A)$.
\end{Def}

The following are some basic properties of the semigroup $\widehat{S}(A)$. To keep the paper short we skip
the proof.
\begin{Prop} \label{prop-S-hat-multi-lambda}
1) If $A$ is a graded algebra of almost $G$-integral type then
$\widehat{S}(A)$ is a strongly non-negative semigroup and hence the multiplicity convex set
$\widehat{\Delta}(A)$ is a convex body.
2) The lattice $\widehat{\Lambda}(A)$ associated to the semigroup $\widehat{S}(A)$ is contained in
the lattice $\widehat{\Lambda}(X)$. It follows that the multiplicity body $\widehat{\Delta}(A)$ is
{parallel to the subspace} $\widehat{\Lambda}_\r(X)$.
3) Suppose $L \in {\bf K}_G(X)$ is such that the Kodaira map $\Phi_L$ gives a birational
isomorphism between $X$ and its image. Then the lattice $\widehat{\Lambda}(A_L)$ coincides
with $\widehat{\Lambda}(X)$.
\end{Prop}

Next we will show that the semigroup $\widehat{S}(A)$ contains information about the
multiplicities of the irreducible representations appearing in $A$ {and hence the name}.
Following Okounkov, we show that there is a natural projection from $\widehat{S}(A)$ onto the weight semigroup
$S(A)$. Write
$$A = \bigoplus_{k \geq 0} \bigoplus_{\lambda \in \Lambda^+} L_{k, \lambda},$$
where $L_{k, \lambda}$ is the $\lambda$-isotypic component in $L_k$, i.e. the sum of all
copies of the irreducible representation $V_\lambda$ in $L_k$.
Let us recall the following well-known facts:
{\it 1) The subspace of $U$-invariants $V_\lambda^U$ in an irreducible representation $V_\lambda$ is
one-dimensional (the highest weight vectors).
2) The dimension of the subspace $L_{k, \lambda}^U$ is equal to the multiplicity $m_{k,\lambda}$
of the highest weight $\lambda$ in the $G$-module $L_k$.}

\begin{Lem} \label{lem-valuation-B-weight}
Let $k>0$ be an integer.
1) Let $f \in L_{k, \lambda}$, $g \in L_{k, \mu}$ be two $B$-eigenfunctions
of weights $\lambda$, $\mu$ respectively and assume that $\lambda \neq \mu$. Then
$v(f) \neq v(g)$.
2) For any value $a \in v(L_k^U \setminus \{0\})$ there is a dominant weight
$\lambda \in \Lambda^+$ and a $B$-eigenfunction $f \in L^U_{k, \lambda}$ such that $v(f) = a$.
\end{Lem}
\begin{proof}
1) By contradiction suppose that $a = v(f) = v(g)$. Consider the leaf
$$F_a = \{ h \mid v(h) \geq a\} / \{ h \mid v(h) > a\}.$$ Since $v$ has one-dimensional leaves,
$F_a$ has dimension $1$. On the other hand,
since the valuation $v$ is $B$-invariant, $F_a$ is a $B$-module. But then $F_a$ can not have
two distinct weights $\lambda$ and $\mu$. The contradiction proves the claim.
2)  We know that
$\dim L_k^U = \sum_{\lambda \in \Lambda^+} \dim L_{k,\lambda}^U$, moreover (since $v$ has one-dimensional leaves) $\dim L_k^U = \#v(L_k^U \setminus \{0\})$ and
$\dim L_{k,\lambda}^U = \#v(L_{k,\lambda}^U \setminus \{0\})$.
But by 1) if $\lambda \neq \mu$ then $v(L_{k,\lambda}^U \setminus \{0\}) \cap
v(L_{k,\mu}^U \setminus \{0\}) = \emptyset$. Thus $v(L_{k}^U \setminus \{0\})$ is equal to $\bigcup_{\lambda \in \Lambda^+} v(L_{k,\lambda}^U \setminus \{0\})$. This proves 2).
\end{proof}

Let $a \in v(L_k^U \setminus \{0\})$ be a value of the valuation. By Lemma
\ref{lem-valuation-B-weight}(2) there exists a unique weight $\lambda(a)$ such that $a = v(f)$ for some
$B$-eigenfunction $f$ of weight $\lambda(a)$.
Consider the map $(k, a) \mapsto (k, \lambda(a))$. It is a surjective additive map
$\widehat{\pi}$ from the semigroup $\widehat{S}(A)$ onto $S(A)$: Let
$f$, $g$ be two $B$-eigenfunctions of weights $\lambda$, $\mu$ respectively. Also let
$v(f) = a$ and $v(g) = b$. Then clearly $fg$ is a $B$-eigenfunction of weight $\lambda+\mu$ and
$v(fg) = a+b$. That is, $\lambda(a+b) = \lambda + \mu = \lambda(a) + \lambda(b)$.
The map $\widehat{\pi}$ then extends to a
linear map from $L(\widehat{S}(A))$ to $L(S(A))$ which by abuse of notation we denote again by $\widehat{\pi}$.

The number of points in the fibre of $\widehat{\pi}$ over a point $(k, \lambda) \in S(A)$
gives the multiplicity of the corresponding highest weight in $L_k$:
\begin{Prop} \label{prop-hat-pi-multiplicity}
{For each $(k, \lambda) \in S(A)$, the number of points in 
$\widehat{\pi}^{-1}(k, \lambda) \cap \widehat{S}(A)$} is equal to the multiplicity
$m_{k, \lambda}$ of the irreducible representation $V_\lambda$ in $L_{k}$.
\end{Prop}
\begin{proof}
From definition, {the number of points in $\widehat{\pi}^{-1}(\lambda, k) \cap \widehat{S}(A)$} is equal
to the number of points in the image of
$L_{\lambda, k}^U \setminus \{0\}$ under the valuation $v$. But
this is equal to the dimension of $L_{k, \lambda}^U$ which in turn is equal to
the $\lambda$-multiplicity of $L_{k}$.
\end{proof}

Let $M(k) = \sum_{\lambda \in \Lambda^+} m_{k, \lambda}$ be the sum of multiplicities in $L_k$ (i.e.
$M(k)$ is the Hilbert function of the graded algebra $A^U$). Then $M(k)$ is equal to the number of
points in the semigroup $\widehat{S}(A)$ at level $k$. Applying Theorem \ref{th-asymp-H_S-vol-Delta} to the
semigroup $\widehat{S}(A)$ we obtain:

\begin{Cor} \label{cor-sum-multi-volume-Okounkov}
Let $A \in {\bf A}_G(X)$ be a graded algebra of almost $G$-integral type.
Let $\widehat{p} = \dim \widehat{\Delta}(A)$ be the dimension of the multiplicity body of $A$.
Then the function $M(k)$ has growth degree $\widehat{p}$, moreover:
$$ \lim_{k \to \infty} \frac{M(k)}{k^{\widehat{p}}} = \Vol_{\widehat{p}}(\widehat{\Delta}(A)),$$
where the volume is normalized with respect to the lattice $\widehat{\Lambda}(A)$.
\end{Cor}

In Section \ref{sec-int-index-G} we will give a formula for the self-intersection index of
an invariant subspace $L$ of rational functions in terms of the integral of an (explicitly defined)
polynomial function over the multiplicity convex body of the algebra $A = A_L$.

\begin{Rem} One can describe the {asymptotic behavior} of the multiplicities
$m_{k, k\lambda}$, as $k \to \infty$, in terms of
the dimension and volume of the fibres $\widehat{\pi}^{-1}(\lambda)$ of the projection
$\widehat{\pi}: \widehat{\Delta}(A) \to \Delta(A)$.
\end{Rem}

\begin{Ex}
As in Example \ref{ex-moment-body-G/B},
let $X = G/B$ and let $A$ be the ring of sections of a $G$-linearized line bundle $\L_\lambda$ associated to
a dominant weight $\lambda$. Let us regard $X$ as a $T$-variety and hence $A$ as a $T$-algebra.
It is shown in \cite{Kiumars-string} that for a choice of a natural $T$-invariant
valuation $v$ on $\c(X)$, the multiplicity body $\widehat{\Delta}_T(A)$ coincides with a so-called string polytope corresponding to the weight $\lambda$ (see Section \ref{subsec-string-preliminary}). {The well-known Gelfand-Cetlin polytopes of $G = \GL(n, \c)$ are special cases of string polytopes.}
\end{Ex}

\subsection{Duistermaat-Heckman measure for graded $G$-algebras}
\label{subsec-Duistermaat}
Let $X$ be an irreducible $G$-variety. In this section we extend the definition of
the Duistermaat-Heckman measure (for projective $G$-varieties and $G$-linearized very ample line bundles)
to graded $G$-algebras $A \in {\bf A}_G(X)$.

Let us recall the Duistermaat-Heckman measure for a Hamiltonian action from symplectic geometry:
Let $K$ be a compact Lie group and $X$ a {compact Hamiltonian $K$-manifold} with the moment map
$\phi: X \to \Lie(K)^*$. Let us denote the moment polytope which is the intersection of
$\phi(X)$ with the positive Weyl chamber by $\Delta(X)$.
Let $\lambda \in \Delta(X)$ be a regular value for the moment map and let
$K_\lambda$ denote its $K$-stabilizer. The reduced space
$X_\lambda = \phi^{-1}(\lambda) / K_\lambda$
is a symplectic manifold with respect to a natural symplectic form.
The Duistermaat-Heckman measure on the polytope $\Delta(X)$ is the measure
$\Vol(X_\lambda) d\gamma$, where $\Vol$ is the symplectic volume and
$d\gamma$ is the Lebesgue measure on $\Delta(X)$ (normalized with respect to {the weight lattice}
$\Lambda$).

Let $(A, \varphi) \in {\bf A}_G(X)$ be a graded algebra of almost $G$-integral type.
Let $p = \dim \Delta(A)$ and $\widehat{p} = \dim \widehat{\Delta}(A)$ be the dimensions of the
moment body $\Delta(A)$ and the multiplicity body $\widehat{\Delta}(A)$ respectively. From Corollary \ref{cor-sum-multi-volume-Okounkov}
we know that $M(k)$,  the sum of multiplicities in $L_k$,
has growth degree equal to $\widehat{p}$.
For each integer $k>0$ consider the measure $d\mu_k$ with finite support defined on the positive Weyl chamber
$\Lambda^+_\r$ by $$d\mu_k = \sum_{\lambda \in \Lambda^+} m_{k, \lambda} \delta_{\lambda/k},$$ where
$\delta_x$ denotes the Dirac measure centered at the point $x$.
Let $d\widehat{\gamma}$ be the Lebesgue measure on the multiplicity body $\widehat{\Delta}(A)$ normalized
with respect to the lattice $\widehat{\Lambda}(A)$ associated to the semigroup $\widehat{S}(A)$.

Consider the linear map $\widehat{\pi}: \widehat{S}(A) \to S(A)$ from Section \ref{subsec-S-hat}. With notations as in Section
\ref{subsec-D-H-semigroup}, the measures $d\mu_k$ are the measures
associated to the pair $(S(A)$, $\widehat{S}(A))$ and the linear map $\widehat{\pi}$.
Applying Theorem \ref{th-D-H-semigroup}, we obtain the following:
\begin{Th}
1) The sequence $d\mu_k / k^{\widehat{p}}$ weakly converges (as $k \to \infty$) to a
(finite) measure $d\mu_A$ supported on the moment convex body $\Delta(A)$.
2) The measure $d\mu_A$ is equal to the push-forward, under $\widehat{\pi}$, of the
Lebesgue measure $d\widehat{\gamma}$ on $\widehat{\Delta}(A)$ to the moment body $\Delta(A)$.
\end{Th}

\begin{Rem}[Duistermaat-Heckman measure for projective $G$-subvarieties] \label{rem-DH-measure-proj-G-var}
Let $(L, \varphi) \in {\bf K}_G(X)$ be an invariant subspace of rational functions with the
Kodaira map $\Phi_L: X \to \p(L^*)$. As usual let $Y_L \subset \p(L^*)$ denote the
closure of the image of $\Phi_L$. It is a closed irreducible $G$-invariant subvariety of
the projective space $\p(L^*)$. Let us assume $Y_L$ is smooth.
As in Remark \ref{rem-moment-polytope-symplectic},
fix a $K$-invariant Hermitian inner product on $L^*$ where $K$ is a
maximal compact subgroup of $G$. This induces a $K$-invariant symplectic structure on $\p(L^*)$ and
hence on $Y_L$. With this $Y_L$ becomes a Hamiltonian $K$-space.
Using the principle of {\it quantization commutes with reduction} one proves that (up to multiplication by a
constant) the
Duistermaat-Heckman measure of the Hamiltonian space $Y_L$ coincides with the measure
$d\mu_{A_L}$ (see \cite[Theorem 6.5]{G-S}).

\begin{Def}[Duistermaat-Heckman measure for $G$-algebras]
Let $(A, \varphi) \in {\bf A}_G(X)$ be a graded algebra of almost $G$-integral type.
In analogy with the case of Hamiltonian spaces, we call $d\mu_A$ the {\it Duistermaat-Heckman measure
associated to the $G$-algebra $A$}.
\end{Def}

In Section \ref{subsec-int-index-G} we will give a formula for the self-intersection index of an invariant
subspace $L$ of rational functions in terms of the integral of an (explicitly defined) polynomial over the
moment body of $A=A_L$ with respect to the Duistermaat-Heckman measure
of the algebra $A$.

\end{Rem}

{Similarly to the moment body}, the multiplicity body also enjoys a superadditivity property.
Let $(A', \varphi'), (A'', \varphi'') \in {\bf A}_G(X)$ be two graded
$G$-algebras. As in Section \ref{subsec-superadd-moment},
we have $(A'^U)(A''^U) \subset (A'A'')^U$.
In general $(A'^U)(A''^U)$ might be strictly
smaller than $(A'A'')^U$ and thus the map $A \mapsto \widehat{\Delta}(A)$ is in general
only superadditive:
\begin{Prop}[Superadditivity for multiplicity body] \label{prop-hat-Delta_G-superadd}
With notation as above, we have
$$\widehat{\Delta}(A') + \widehat{\Delta}(A'') \subset \widehat{\Delta}(A'A'').$$
In particular, if $(L', \varphi'), (L'', \varphi'') \in K_G(X)$ are invariant subspaces then we have
$$\widehat{\Delta}(A_{L'}) + \widehat{\Delta}(A_{L''}) \subset \widehat{\Delta}(A_{L'L''}).$$
\end{Prop}
\begin{proof}
The inclusion $(A'^U)(A''^U) \subset (A'A'')^U$ implies that
that the levelwise addition of $\widehat{S}(A')$ and $\widehat{S}(A'')$ is contained in
$\widehat{S}(A'A'')$. The superadditivity immediately follows from this inclusion.
\end{proof}

Recall that by Proposition \ref{prop-S-hat-multi-lambda}(1),
the multiplicity body $\widehat{\Delta}(A)$ is {parallel to the
subspace} $\widehat{\Lambda}_\r(X)$. Also $\widehat{r}=\widehat{r}(X)$ denotes the dimension of $\widehat{\Lambda}_\r(X)$.
Applying the superadditivity (Proposition \ref{prop-hat-Delta_G-superadd})
and the classical Brunn-Minkowski inequality we obtain the following:

\begin{Cor}[Brunn-Minkowski inequality for D-H measure of algebras] \label{cor-Brunn-Mink-Okounkov-algebra}
Let $(A', \varphi)$, $(A'', \varphi'') \in {\bf A}_G(X)$ be two graded
$G$-algebras. Also assume that
$\widehat{\Lambda}(A') = \widehat{\Lambda}(A'') = \widehat{\Lambda}(X)$.
Let $d\mu_{A'} = f_{A'} d\gamma$, $d\mu_{A''} = f_{A''} d\gamma$ and
$d\mu_A = f_A d\gamma$ denote the Duistermaat-Heckman functions for the algebras $A'$, $A''$ and $A=A'A''$ respectively.
Here $d\gamma$ is the Lebesgue measure on $\widehat{\Lambda}(X)_\r$ (or its parallel shifts)
normalized with respect to $\widehat{\Lambda}(X)$.
Then for $\lambda' \in \Delta(A')$, $\lambda'' \in \Delta(A'')$ we have
$$ f_{A'}(\lambda')^{1/\widehat{r}} + f_{A''}(\lambda'')^{1/\widehat{r}} \leq f_A(\lambda' +\lambda'')^{1/\widehat{r}}.$$
\end{Cor}
In particular, applied to to the homogeneous coordinate rings, the above corollary implies
a Brunn-Minkowski inequality for the Duistermaat-Heckman measures of projective $G$-varieties.

\subsection{Fujita approximation for Duistermaat-Heckman measure} \label{subsec-Fujita}
The Fujita approximation theorem in the theory of divisors states that the so-called volume of a
big divisor can be approximated arbitrarily closely by the self-intersection numbers of very ample
divisors (see \cite{Fujita}, \cite{Lazarsfeld}).
In \cite{Lazarsfeld-Mustata} it is shown that this theorem can be deduced from a statement
about semigroups of integral points and it is extended to a large class of graded linear systems.
Motivated by \cite{Lazarsfeld-Mustata},
in \cite{Askold-Kiumars-Newton-Okounkov} an abstract version of this theorem is proved for general semigroups
of integral points. This then gives a Fujita type approximation for graded algebras which
in turn implies {a slight extension} of the Fujita approximation theorem of Lazarsfeld-Mustata to
arbitrary graded linear systems.

Now we apply Theorem \ref{th-Fujita-D-H-semigroup} to prove a version of {the Fujita approximation theorem} for the Duistermaat-Heckman measures of graded $G$-algebras. That is, we prove that
the Duistermaat-Heckman measure of a graded $G$-algebra can be approximated arbitrarily closely by the
Duistermaat-Heckman measures of the $G$-algebras of type $A_L$ for finite dimensional invariant subspaces
$L$. This will follow from the analogous statement (Theorem \ref{th-Fujita-D-H-semigroup}) for semigroups.
In the case of algebra of sections of $G$-line bundles, it implies that the Duistermaat-Heckman measure of a $G$-line bundle can be approximated
arbitrarily closely by that of very ample $G$-line bundles. (The Duistermaat-Heckman measures of very ample line bundles are the usual
ones defined by symplectic geometry, i.e. regarding $X$ as a $K$-Hamiltonain space, see Remark \ref{rem-DH-measure-proj-G-var} .)

Take an integer $k>0$ such that $L_k \neq \{0\}$.
Consider the graded $G$-algebra $A_{L_k}$ associated to $L_k$.
Let $d\rho_k$ be the Duistermaat-Heckman measure associated to $A_{L_k}$. In other words,
$d\rho_k$ is the Duistermaat-Heckmann measure of the projective $G$-subvariety $Y_{L_k} \subset \p(L_k^*)$, where $Y_{L_k}$ is the
closure of the image of the Kodaira map of $L_k$.
Let $\widehat{m}_{k,\ell, \lambda}$ be the multiplicity of the irreducible representation $V_\lambda$ in $(L_k)^\ell$, i.e.
the $\ell$-th subspace of the algebra $A_{L_k}$. Then:
$$d\rho_k = \lim_{\ell \to \infty} (1/\ell^{\widehat{p}}) \sum_{\lambda \in \Lambda^+}
\widehat{m}_{k, \ell, \lambda} \delta_{\lambda/\ell}.$$
The Duistermaat-Heckman measure $d\rho_k$ is supported on the convex polytope $\Delta(A_{L_k})$ (which is contained
in the convex body $k \Delta(A)$).

Let $O_{1/k}: \Lambda_\r \to \Lambda_\r$ denote the multiplication by the scalar $1/k$.
\begin{Th}[Fujita approximation type theorem for Duistermaat-Heckman measure of algebras] \label{th-FUjita-D-H-G-algebra}
Let $(A, \varphi) \in {\bf A}_G(X)$ be a graded algebra of almost $G$-integral type.
Then, as $k \to \infty$, the measures $O^*_{1/k}(d\rho_k)/k^{\widehat{p}}$ converge weakly to the
Duistermaat-Heckman measure $d\mu_A$
associated to the $G$-algebra $A$. Here $O^*_{1/k}$ denotes the push-forward of the measure $d\rho_k$ on
$\Delta(A_{L_k})$ to the convex body $\Delta(A)$.
\end{Th}
\begin{proof}
The claim follows from Theorem \ref{th-Fujita-D-H-semigroup} applied to
$\widehat{\pi}: \widehat{S}(A) \to S(A)$.
\end{proof}


\section{String convex body of a $G$-algebra} \label{sec-string}
Let $A \in {\bf A}_G(X)$ be a graded $G$-algebra. As in Section \ref{sec-multi-body}, fix a $B$-invariant valuation $v$ on $\c(X)$ (with one-dimensional leaves and values in $\z^n$). To the $G$-algebra $A$ we associated the multiplicity convex body $\widehat{\Delta}_G(A)$, which is constructed out of the value semigroup of $v$ on the subalgebra $A^U$ .
We can also consider $A$ as a $T$-algebra and consider the convex body
$\widehat{\Delta}_T(A)$ which is constructed out of the value semigroup of $v$ on the whole $A$. In general, since
$A^U \subset A$ we have $\widehat{\Delta}_G(A) \subset \widehat{\Delta}_T(A)$. Moreover, for each $k >0$, the number of
points in $\widehat{S}_G(A)$ at level $k$ is equal to the sum of multiplicities of different highest weights in $L_k$, while
the number of points in $\widehat{S}_T(A)$ at level $k$ is equal to the dimension of $L_k$ and hence $\widehat{\Delta}_T(A)$ is
responsible for the {asymptotic growth} of the Hilbert function of $A$. In this section {we give
a more} canonical construction of a convex body $\widetilde{\Delta}(A)$ {which resembles $\widehat{\Delta}_T(A)$}, and projects onto the multiplicity body $\widehat{\Delta}_G(A)$. In particular,
$\widetilde{\Delta}(A)$ is responsible for the {asymptotic growth} of the Hilbert function of $A$. For this we use the
so-called string polytopes of Littelmann-Berenstein-Zelevinsky associated to irreducible representations of a reductive group.

\subsection{String polytopes for irreducible representations} \label{subsec-string-preliminary}
In this section we recall the definition of string polytopes for a reductive group $G$.

Consider the algebra
$$ A = \c[G]^U$$ of regular functions {on $G$ that are invariant} under the
right multiplication by $U$. The group $G \times T$ acts on $A$
where $G$ acts on the left and $T$ acts on the right, since it
normalizes $U$. It is well-known that as a $G \times T$-module $A$ decomposes into: $$ A \cong \bigoplus_{\lambda
\in \Lambda^+} V_\lambda^*,$$ where $V_\lambda^* = V_{\lambda^*}$ is
the irreducible representation with highest weight $\lambda^* = -w_0(\lambda)$, and
$T$ acts on each $V_\lambda^*$ via the character $\lambda$.

The vector space $A$ has a remarkable basis $\B = (b_{\lambda,
\phi})$, usually called the {\it dual canonical basis}, such that each $b_{\lambda, \phi}$ is an eigenvector of $T \times T \subset G \times T$, of weight $\lambda$ for the right
$T$-action. For fixed $\lambda$, the vectors $b_{\lambda,\phi}$ form
a basis for $V_\lambda^*$.
For a reduced decomposition of $w_0$, the
longest element of $W$, one can define a parametrization of $\B$
called the {\it string parametrization} \cite{Littelmann, B-Z2}.
Recall that an $N$-tuple of simple reflections
$$\s = (s_{i_1}, s_{i_2}, \ldots, s_{i_N})$$ is a {\it reduced
decomposition} for $w_0$ if $w_0 = s_{i_1}s_{i_2}\cdots s_{i_N},~N =
\ell(w_0)$. The string parametrization associated to $\s$ is an
injective map
$$\iota_{\s}: \B \to \Lambda^+ \times \n^N,$$
$$ \iota(b_{\lambda,\phi}) = (\lambda, t_1, \ldots, t_N).$$
The string parameters have to do with the weight of a basis element
as an eigenvector for the $T\times T$-action: the weight of
$b_{\lambda, \phi} \in \B$ for the left $T$-action is $$-\lambda +
t_1\alpha_{i_1} + \cdots + t_N\alpha_{i_N}.$$

A remarkable property of the string parameterization is that its image consists of all the
integral points in a certain rational convex polyhedral cone
$\mathcal{C}$ in $\Lambda_{\r} \times \r^N$ (\cite{Littelmann}).
\begin{Def} \label{def-string-polytope}
The string polytope $\Delta_{\s}(\lambda)$ is the
polytope in $\r^N$ obtained by slicing the cone $\mathcal{C}$ at
$\lambda$, that is $$\Delta_{\s}(\lambda) = \{(t_1, \ldots, t_N)
\mid (\lambda, t_1, \ldots, t_N) \in \mathcal{C} \}.$$
\end{Def}
Note that: 1) $\Delta_{\s}(\lambda)$ is defined for any $\lambda \in \Lambda_\r^+$.
2) From the definition it follows that $\Delta_\s(k\lambda) =
k\Delta_\s(\lambda)$ for any positive integer $k$.
3) The fact that $\mathcal{C}$ is a convex cone implies that for $\lambda_1$, $\lambda_2 \in \Lambda^+_\r$
we have $\Delta_\s(\lambda_1) + \Delta_\s(\lambda_2) \subset \Delta_\s(\lambda_1 + \lambda_2)$.

By what was said above, when $\lambda$ is a dominant weight,
the lattice points in $\Delta_{\s}(\lambda)$, i.e. the points
in $\Delta_{\s}(\lambda) \cap \z^N $, are in
bijection with the elements of the basis $b_{\lambda, \phi}$ for
$V_\lambda^*$ (and hence in bijection with the basis for
$V_\lambda$). Thus,
\begin{equation} \label{equ-dimV_lambda} \#(\Delta_{\s}(\lambda)
\cap \z^N) = \dim(V_\lambda)
\end{equation}

Let $v_\lambda \in V_\lambda$ be a highest weight vector and
$P_\lambda$ the parabolic subgroup associated to the
weight $\lambda$, that is, $P_\lambda$ is the stabilizer of the
point $[v_\lambda] \in \p(V_\lambda)$. Then we have an embedding
$i: G/P_\lambda \hookrightarrow \p(V_\lambda)$, given by
$gP_\lambda \mapsto g \cdot [v_\lambda]$. Let
$L_\lambda = i^*(\mathcal{O}(1))$ be the line bundle on $G/P_\lambda$ induced by this embedding.
By Borel-Weil one knows that
for $k>0$, $H^0(X, L_\lambda^{\otimes k}) \cong V_{k\lambda}^*$ as $G$-modules.
Put $m=\dim(G/P_\lambda)$. From
(\ref{equ-dimV_lambda}) it follows that the degree of $G/P_\lambda$, as a subvariety of
$\p(V_\lambda)$ is equal to $m!\Vol_m(\Delta_{\s}(\lambda))$.

\begin{Rem} \label{rem-GC}
The Gelfand-Cetlin polytopes \cite{G-C, B-Z1} are special cases of the string
polytopes. {More precisely, let $G=\GL(n, \c)$. The Weyl group
is $W = S_{n}$. Let us take the nice reduced decomposition}
$$w_0 = (s_1)(s_2s_1)(s_3s_2s_1) \cdots (s_{n-1}\cdots s_1)$$
for $w_0$, where $s_i$ denotes the transposition exchanging $i$ and
$i+1$. Then $\Delta_{\s}(\lambda)$ can be identified with the well-known
Gelfand-Cetlin polytope corresponding to $\lambda$. Similarly, when
$G = \SP(2n, \c)$ or $\SO(n, \c)$, for a similar choice of a reduced
decomposition, one can recover the Gelfand-Cetlin polytopes as the
string polytopes \cite{Littelmann}.
\end{Rem}

\subsection{String convex body of a $G$-algebra} \label{subsec-string-semigroup-G}
Let $(A, \varphi)$ be a graded $G$-algebra. Using the string polytopes we now define a semigroup lying over the
multiplicity semigroup $\widehat{S}(A)$ which encodes information both about multiplicities and
the dimensions of the isotypic components.

Fix a reduced decomposition ${\s}$ for the longest element $w_0$ in the Weyl group.
{Let $a \in v(L_k^U \setminus \{0\})$ be a value of the valuation. Recall that we associate a weight $\lambda(a)$ to $a$. The linear projection $\widehat{\pi}: \widehat{S}(A) \to S(A)$
is then defined by $(k, a) \mapsto (k, \lambda(a))$ 
(paragraph before Proposition \ref{prop-hat-pi-multiplicity}).}

\begin{Def}
Define the set $\widetilde{S}_G(A) \subset \z_{\geq 0} \times \z^{n+N}$ by
$$\widetilde{S}_G(A) = \{ (k, a, b) \mid (k, a) \in \widehat{S}(A),~ b
\in \Delta_{\s}(\lambda(a)) \cap \z^N \},$$ where $\Delta_{\s}(\lambda)$ is the string
polytope associated to the weight $\lambda$ and the reduced decomposition ${\s}$.
When there is no ambiguity we will write $\widetilde{S}(A)$ instead of $\widetilde{S}_G(A)$.
\end{Def}

The map $(k, a, b) \mapsto (k, a)$ induces a surjective map
$\widetilde{\pi}: \widetilde{S}(A) \to \widehat{S}(A)$.

\begin{Prop} \label{prop-S-tilde-dim-lambda}
1) The set $\widetilde{S}_G(A)$ is a semigroup (under addition).
2) For every $(k, a) \in \widehat{S}(A)$, the number of points in
$\widetilde{\pi}^{-1}(k, a)$ is equal to the dimension of the irreducible representation
$V_\lambda$, where $\lambda = \lambda(a)$.
3) Consider the linear map $\widehat{\pi} \circ \widetilde{\pi}: \widetilde{S}(A) \to S(A)$.
For every $(k, \lambda) \in S(A)$, the number of points in
$(\widehat{\pi} \circ \widetilde{\pi})^{-1}(k, \lambda)$ is equal to the dimension of
$\lambda$-isotypic component $L_{\lambda, k}$ in $L_k$.
If $A$ is an algebra of almost $G$-integral type then the semigroup
$\widetilde{S}(A)$ is a strongly non-negative semigroup and hence its Newton-Okounkov convex set
is a convex body.
\end{Prop}
\begin{proof}
1) We know that for any two $\lambda_1$, $\lambda_2 \in \Lambda^+_\r$,
$\Delta_\s(\lambda_1) + \Delta_\s(\lambda_2) \subset \Delta_{\s}(\lambda_1 + \lambda_2)$
(the paragraph after Definition \ref{def-string-polytope}). This implies
that $\widetilde{S}(A)$ is a semigroup. 2) The number of points in $\widetilde{\pi}^{-1}(k, a)$
is equal to the number of integral points in the string polytope $\Delta_{\s}(\lambda)$ (where
$\lambda = \lambda(a)$), and the latter is equal to $\dim V_\lambda$.
3) By Proposition \ref{prop-S-hat-multi-lambda}(1)
we know that the semigroup $\widehat{S}(A)$ is a strongly non-negative semigroup.
Also the cone $\mathcal{C}_{\s}$ is strongly convex. These two facts imply the claim.
\end{proof}

\begin{Def}
We denote the Newton-Okounkov convex set of the semigroup $\widetilde{S}(A)$
by $\widetilde{\Delta}(A)$ and call it the {\it string convex set of $(A, \varphi)$}.
By above, when $A$ is of almost $G$-integral type, $\widetilde{\Delta}(A)$ is a convex body.
\end{Def}

\begin{Prop}[Superadditivity of the string body] \label{prop-superadditive-string-body}
Let $(A', \varphi'), (A'', \varphi'') \in {\bf A}_G(X)$ be two graded $G$-algebras.
Then we have: $$\widetilde{\Delta}(A') + \widetilde{\Delta}(A'') \subset \widetilde{\Delta}(A).$$
In particular, if $(L', \varphi'), (L'', \varphi'') \in K_G(X)$ are invariant subspaces then:
$$\widetilde{\Delta}(A_{L'}) + \widetilde{\Delta}(A_{L''}) \subset \widetilde{\Delta}(A_{L'L''}).$$
\end{Prop}
\begin{proof}
It follows immediately from the superadditivity of the string polytope (the paragraph after Definition \ref{def-string-polytope})
and the superadditivity of the multiplicity body
(Proposition \ref{prop-hat-Delta_G-superadd}).
\end{proof}


{Finally, analogously} to the weight lattice $\Lambda(X)$ and the multiplicity lattice $\widehat{\Lambda}(X)$,
we can associate a sublattice $\widetilde{\Lambda}(X)$ of $\z^{n+N}$ to the variety $X$ which is the largest
possible lattice that can appear as $\widetilde{\Lambda}(A)$ for a $G$-algebra $A$: Define the
lattice $\widetilde{\Lambda}(X)$ to be the lattice generated by all the $(a, b)$ where
$a \in \widehat{\Lambda}(X)$ {is such that $\lambda(a)$ is dominant, and $b \in \Delta_{\s}(\lambda(a)) \cap \z^N$.}
The next proposition follows from the definition and Proposition \ref{prop-S-hat-multi-lambda}(2) and (3).

\begin{Prop} \label{prop-tilde-Lambda-A_L}
1) Let $(A, \varphi) \in {\bf A}_G(X)$ be a $G$-algebra. Then the
lattice $\widetilde{\Lambda}(A)$ is contained in $\widetilde{\Lambda}(X)$.
2) Suppose $(L, \varphi) \in {\bf K}_G(X)$ is an invariant subspace such that the
Kodaira map $\Phi_L$ is a birational isomorphism between $X$ and its image. Then
$\widetilde{\Lambda}(A_L)$ coincides with $\widetilde{\Lambda}(X)$.
\end{Prop}

\section{Self-intersection index of invariant subspaces} \label{sec-int-index-G}
In this section we give formulae for the growth of the Hilbert function of
a graded $G$-algebra in terms of the convex bodies $\Delta(A)$, $\widehat{\Delta}(A)$ and
$\widetilde{\Delta}(A)$. When $A = A_L$, for an invariant subspace $(L, \varphi) \in {\bf K}_G(X)$,
this implies formulae for the self-intersection index of $L$.

\subsection{Formulae for the growth of Hilbert function and self-intersection index} \label{subsec-int-index-G}
Let $X$ be an irreducible $G$-variety of dimension $n$.
Let $(A, \varphi) \in {\bf A}_G(X)$ be a graded algebra of almost $G$-integral type,
and let $H_A$ denote the Hilbert function of $A$.

We start with the largest of the three convex bodies namely $\widetilde{\Delta}$.
Let $\widetilde{S}_k$ denote the points at level $k$ in the semigroup $\widetilde{S}$.
As in Section \ref{subsec-string-semigroup-G} we have $\dim L_k = \# \widetilde{S}_k.$
Applying Theorem \ref{th-asymp-H_S-vol-Delta} to the semigroup $\widetilde{S}$ we obtain:

\begin{Th} \label{th-Hilbert-growth-G-alg-tilde-S}
1) The growth degree $q$ is equal to the dimension of the convex body $\widetilde{\Delta}$.
2) The $q$-th growth coefficient $a_q = \lim_{k \to \infty} H_A(k)/k^q$
is equal to $\Vol_q(\widetilde{\Delta})$ where the
volume is normalized with respect to the lattice $\widetilde{\Lambda}(A)$ associated to the semigroup
$\widetilde{S}$.
\end{Th}

Let $\lambda \in \Lambda^+$ be a dominant weight. By the Weyl dimension formula we have
$$F(\lambda) = \dim V_\lambda = \prod_{\alpha \in R^+}
\langle \lambda+\rho, \alpha \rangle / \langle \rho, \alpha
\rangle,$$
where $R^+$ is the set of
positive roots and $\rho$ is half the sum of positive roots.
Consider $F$ as a function on the lattice of weights $\Lambda(A)$ (consisting of
all the differences $\lambda-\mu$ for all $k>0$ and $\lambda, \mu \in S_k$).
Then the homogeneous component of the highest degree of $F$, as a function on the subspace
$\Lambda_\r(A)$ spanned by the lattice $\Lambda(A)$, is given by:
\begin{equation} \label{equ-f}
f(\lambda) = \prod_{\alpha \in R^+ \setminus E}
\frac{\langle \lambda, \alpha \rangle}{\langle \rho,
\alpha\rangle}.
\end{equation}
Here $E$ is the set of positive roots which are orthogonal to the moment polytope $\Delta$.

Let $\widehat{\pi}^*F$ and $\widehat{\pi}^*f$ denote the pull-backs of
$F$ and $f$ to the semigroup $\widehat{S}$ via the projection $\widehat{\pi}: \widehat{S} \to S$ respectively.
Applying Theorem \ref{th-asymp-H_S-vol-Delta} to the semigroup $\widehat{S}$ and the polynomial $\widehat{\pi}^*F$ we get:

\begin{Cor} \label{cor-Hilbert-growth-G-alg-hat-S}
The $q$-th growth coefficient $a_q$ of the Hilbert function $H_A$ is equal to the
integral $$\int_{\widehat{\Delta}} \widehat{\pi}^*f d\widehat{\gamma},$$
where $d\widehat{\gamma}$ is the Lebesgue measure
on the real span of the convex body $\widehat{\Delta}$ normalized with respect to the lattice
$\widehat{\Lambda}(A)$.
\end{Cor}

Finally we can give a formula for the growth coefficient $a_q$ as an integral
over the moment convex body $\Delta$. Recall that $d\mu$ denotes the Duistermaat-Heckman measure
for the $G$-algebra $A$.

\begin{Cor} \label{cor-Hilbert-growth-G-alg-S}
The $q$-th growth coefficient $a_q$ of the Hilbert function $H_A$ is equal to the
integral $$\int_{\Delta} f d\mu.$$
\end{Cor}

Using the above results we can get formulae for the self-intersection number of an
invariant subspace.
Let $(L, \varphi) \in {\bf K}_G(X)$ be an invariant subspace of rational functions on $X$.
Let $A_L = \bigoplus_k L^k$ be the algebra associated to $L$. Also, as in Example \ref{ex-ring-sections-G-alg}(2),
let $A = \ol{A_L}$ denote the integral closure of this algebra regarded as a subalgebra of $F[t]$.
As above, let $\Delta$, $\widehat{\Delta}$ and $\widetilde{\Delta}$ denote respectively the moment body, multiplicity body and string body
of the algebra $A$. Note that since $A_L$ is finitely generated, $A$ is also finitely generated and hence
$\Delta$ is a polytope.

From Hilbert's theorem on the dimension and degree of a projective variety, it follows that the self-intersection
index $[L, \ldots, L]$ is equal to $n!$ times the coefficient of degree $n$ in the Hilbert polynomial of
the algebra $A$ (see \cite[Section 4]{Askold-Kiumars-Newton-Okounkov}).

Suppose {the Kodaira map} $\Phi_L$ has finite mapping degree, that is, $\dim(Y_L) = n$. Then one can show that
for any large enough $k$, the Kodaira map $\Phi_{L_k}$ is a birational isomorphism, where $L_k$ denotes the $k$-th homogeneous
piece of $A$ (see \cite[Section 4.3]{Askold-Kiumars-Newton-Okounkov}). From this it follows that
the lattices corresponding to the algebra $A$ are the largest possible, that is, $\Lambda(A) = \Lambda(X)$,
$\widehat{\Lambda}(A) = \widehat{\Lambda}(X)$ and $\widetilde{\Lambda}(A) = \widetilde{\Lambda}(X)$. Putting these together we obtain the following:


\begin{Th} \label{th-self-int-index-G-subspace-vol}
We have the following formula for the self-intersection index of an invariant subspace $(L, \varphi)$:
$$[L, \ldots, L] = n! \Vol_n(\widetilde{\Delta}),$$
where $\Vol_n$ is the Lebesgue measure in $\widetilde{\Lambda}_\r(X) = \widetilde{\Lambda}(X) \otimes \r$
normalized with respect to the lattice $\widetilde{\Lambda}(X)$.
\end{Th}
\begin{proof}
It follows from Theorem \ref{th-Hilbert-growth-G-alg-tilde-S}, as in the proof of \cite[Theorem 4.12]{Askold-Kiumars-Newton-Okounkov}.
\end{proof}

\begin{Cor} \label{cor-self-int-index-G-subspace-integ}
\noindent 1)
$$[L, \ldots, L] = n!\int_{\widehat{\Delta}} \widehat{\pi}^*f d\widehat{\gamma},$$
where the measure $d\widehat{\gamma}$ is the Lebesgue measure in $\widehat{\Lambda}_\r(X)$
normalized with respect to the lattice $\widehat{\Lambda}(X)$.\\

\noindent 2)
$$[L, \ldots, L] = n! \int_{\Delta} f d\mu,$$
where the measure $d\mu$ is the Duistermaat-Heckman measure of the algebra $\ol{A_L}$.\\
\end{Cor}
\begin{proof}
Follows directly from Theorem \ref{th-self-int-index-G-subspace-vol}.
\end{proof}

We have analogous statements for the self-intersection index of divisors on a projective
$G$-variety $X$.

\begin{Cor} \label{cor-self-int-index-divisor}
Let $X$ be a {normal} projective $G$-variety of dimension $n$ and let
$\mathcal{L}$ be a $G$-linearized very ample line bundle on $X$. Let $D$ be a
divisor of $\mathcal{L}$ and let $R(D)$ be the corresponding algebra of sections
regarded as a $G$-algebra (as in Example \ref{ex-ring-sections-G-alg}). Then the self-intersection index $D^n$ of the
divisor $D$ is equal to $$n!\int_{\Delta} fd\mu,$$ where $d\mu$ is the Duistermaat-Heckman measure.
\end{Cor}

\subsection{Case of a spherical variety} \label{subsec-spherical}
A $G$-variety $X$ is called {\it spherical} if a Borel subgroup (and hence every Borel subgroup)
has a dense orbit. Some authors require that a spherical variety be normal. Here we do not
need the normality assumption.

\begin{Rem}
\noindent 1) When $G=T$ is a torus, spherical varieties are exactly toric varieties. \\

\noindent 2) By Bruhat decomposition, partial flag varieties $G/P$ are spherical.\\

\noindent 3) Again by Bruhat decomposition, $G$ is a spherical variety for the action of
$G \times G$ given by multiplication from left and right.\\
\end{Rem}

It is well-known that the spaces of sections of $G$-line bundles over spherical varieties are
multiplicity-free $G$-modules. The following is the analogous statement for the invariant
subspaces of rational functions.
\begin{Prop}
If $X$ is spherical and $(L, \varphi) \in \K$ then $L$ is a multiplicity-free $G$-module.
\end{Prop}
\begin{proof} For a dominant weight $\lambda$ let $f, g \in L$ be two $B$-eigenvectors with
weight $\lambda$ in the $G$-module $L$. Then $f/g \in \c(X)$ is a $B$-invariant function,
for the natural (non-twisted) action of $G$ on $\c(X)$. But as $X$ is spherical it has a dense $B$-orbit
which then implies that $f/g$ is a constant function. Thus $f$ is a scalar multiple of $g$ which shows that
every highest weight representation $V_\lambda$ appears in $L$ with multiplicity at most $1$.
\end{proof}

Since invariant subspaces of functions over spherical varieties are always multiplicity-free,
we observe that, when $X$ is spherical, for any algebra $(A, \varphi) \in {\bf A}_G(X)$ the multiplicity
body $\widehat{\Delta}(A)$ coincides with the moment body $\Delta(X)$. In this case,
Theorem \ref{th-self-int-index-G-subspace-vol} and Corollary \ref{cor-self-int-index-G-subspace-integ}
get nicer and more explicit forms.

\begin{Cor}[Self-intersection index of invariant subspaces for spherical varieties] \label{cor-self-int-index-subspace-spherical}
Let $X$ be a spherical variety of dimension $n$ and let $(L, \varphi) \in \K$ be an
invariant subspace of rational functions. We have
$$[L, \ldots, L] = n! \Vol_n(\widetilde{\Delta}(\ol{A_L})) = n!\int_{\Delta(\ol{A_L})} fd\gamma,$$
where $d\gamma$ is the Lebesgue measure (normalized with respect to
$\Lambda(X)$) and $f$ is as in the paragraph preceding Corollary \ref{cor-Hilbert-growth-G-alg-hat-S}.
\end{Cor}

\begin{Cor}[Self-intersection index of divisors for spherical varieties]
\label{cor-self-int-index-divisor-spherical}
Let $X$ be a {normal} projective spherical variety of dimension $n$ and let
$\mathcal{L}$ be a $G$-linearized very ample line bundle on $X$. Let $D$ be a
divisor of $\mathcal{L}$ and let $R(D)$ be the corresponding algebra of sections
regarded as a $G$-algebra (as in Example \ref{ex-ring-sections-G-alg}).
Then the self-intersection index $D^n$ of the divisor $D$ is equal to:
$$n!\Vol_n(\widetilde{\Delta}(R(D))) = n!\int_{\Delta(R(D))} fd\gamma,$$
where $d\gamma$ and $f$ are as above.
\end{Cor}

\begin{Rem} The algebras $R(D)$ and $\ol{A_L}$ are both finitely generated and hence
both their moment bodies are convex polytopes. Also since for spherical varieties,
the multiplicity body coincide with the moment body, it follows that their string convex bodies are also
convex polytopes. This makes the formulae for the growth of Hilbert functions and self-intersection indices
much more concrete.
\end{Rem}

For the case of $X = G$, with the left-right $(G \times G)$-action, Corollary \ref{cor-self-int-index-subspace-spherical}
is due to B. Kazarnovskii (see \cite{Kazarnovskii}). Corollary \ref{cor-self-int-index-divisor-spherical}
is due to M. Brion (see \cite{Brion1}).


\bibliographystyle{amsplain}



\end{document}